\documentclass[11pt]{amsart}
\usepackage{amsmath,amssymb,amsthm,tikz}

\RequirePackage{color}
\RequirePackage[colorlinks, urlcolor=my-blue,linkcolor=my-red,citecolor=my-green]{hyperref}
\definecolor{my-blue}{rgb}{0.0,0.0,0.6}
\definecolor{my-red}{rgb}{0.5,0.0,0.0}
\definecolor{my-green}{rgb}{0.0,0.5,0.0}
\definecolor{nicos-red}{rgb}{0.75,0.0,0.0}
\definecolor{light-gray}{gray}{0.6}
\definecolor{really-light-gray}{gray}{0.8}
\definecolor{sussexg}{rgb}{0.0,0.5,0.5}
\definecolor{sussexp}{rgb}{0.5,0.0,0.5}

\usepackage[numbers,sort]{natbib}

\usepackage{nicefrac}

\newtheorem{theorem}{\sc Theorem}[section]
\newtheorem{lemma}[theorem]{\sc Lemma}
\newtheorem{proposition}[theorem]{\sc Proposition}
\newtheorem{corollary}[theorem]{\sc Corollary}

\newtheorem{definition}[theorem]{\it Definition}
\numberwithin{equation}{section}
\theoremstyle{remark}
\newtheorem{remark}[theorem]{Remark}

\usepackage{graphicx}
\usepackage{amsmath,tikz,color,calc,amssymb, tikz, calc, tikz-cd, dsfont,bm,bbm}

\numberwithin{equation}{section}
\usepackage[hmargin=2.5cm,vmargin={3 cm,3 cm}]{geometry}
\allowdisplaybreaks

\newcommand{\be}{\begin{equation}}
\newcommand{\ee}{\end{equation}}

\newcommand{\beqar}{\begin{eqnarray*}}
\newcommand{\eeqar}{\end{eqnarray*}}
\newcommand{\bea}{\begin{eqnarray*}}
\newcommand{\eea}{\end{eqnarray*}}

\newcommand{\beqarl}{\begin{eqnarray}}
\newcommand{\eeqarl}{\end{eqnarray}}

\newcommand{\R}{\mathbb{R}}
\newcommand{\doubleint}{ \int_{\R_+}\int_{\R_+}}

\newcommand{\la}{\langle}
\newcommand{\ra}{\rangle}
\newcommand{\NN}{\mathbb{N}}
\newcommand{\E}{\mathbb{E}}
\newcommand{\lp}{\left(}
\newcommand{\rp}{\right)}

\def\bE{\mathbb{E}}
\def\bN{\mathbb{N}}
\def\bP{\mathbb{P}}

\def\bR{\mathbb{R}}
\def\bZ{\mathbb{Z}}

 \def\Z{\bZ} 
 
\def\R{\bR}
\def\N{\bN}
\def\y{\mathbf{y}}
\def\Y{\mathbf{Y}}
\def\z{\mathbf{z}}
\def\e{\varepsilon}

\def\e{\varepsilon}
 
\def\E{\bE}
\def\P{\bP} 

\definecolor{darkgreen}{rgb}{0.0,0.5,0.0}
\definecolor{darkblue}{rgb}{0.0,0.0,0.3}
\definecolor{nicosred}{rgb}{0.65,0.1,0.1}
\definecolor{light-gray}{gray}{0.7}
\allowdisplaybreaks[1]
\RequirePackage{datetime} 

\begin{document}

\allowdisplaybreaks

\usdate
\title[Hydrodynamics for a wealth model]
{Continuum and thermodynamic limits for a simple random-exchange model}

\author[B.~D\"uring]{Bertram D\"uring}
\address{Bertram D\"uring\newline University of Sussex\\ Department of  Mathematics \\ Falmer Campus\\ Brighton BN1 9QH\\ UK.}
\email{bd80@sussex.ac.uk}
\urladdr{http://users.sussex.ac.uk/~bd80} 

\author[N.~Georgiou]{Nicos Georgiou}
\address{Nicos Georgiou\newline University of Sussex\\ Department of  Mathematics \\ Falmer Campus\\ Brighton BN1 9QH\\ UK.}
\email{N.Georgiou@sussex.ac.uk}
\urladdr{http://www.sussex.ac.uk/profiles/329373} 

\author[S.~Merino-Aceituno]{Sara Merino-Aceituno}
\address{Sara Merino-Aceituno\newline University of Vienna\\ faculty of Mathematics \\ Oscar-Morgenstern-Platz 1\\ 1090 Vienna\\ Austria, and\newline 
University of Sussex\\ Department of  Mathematics \\ Falmer Campus\\ Brighton BN1 9QH\\ UK.}
\email{sara.merino@univie.ac.at}
\urladdr{https://saramerinoaceituno.wordpress.com} 

\author[E.~Scalas]{Enrico Scalas}
\address{Enrico Scalas\newline University of Sussex\\ Department of  Mathematics \\ Falmer Campus\\ Brighton BN1 9QH\\ UK.}
\email{E.Scalas@sussex.ac.uk}
\urladdr{http://www.sussex.ac.uk/profiles/330303} 

\thanks{B. D\"uring has been supported by the Leverhulme Trust research project grant: Novel discretisations for higher-order nonlinear PDE (RPG-2015-69). \\
\indent N. Georgiou was partially supported by the EPSRC First Grant EP/P021409/1: The flat edge in last passage percolation. \\
\indent S. Merino-Aceituno is supported by the Vienna Science and Technology Fund (WWTF) with a Vienna Research Groups for Young Investigators, grant VRG17-014. \\
\indent E. Scalas has been supported by the JSPS Invitational Fellowship S18099. }

\keywords{Wealth distribution, mean-field limits, functional limits, Markov chains, kinetic equations, partitions of integers}
\subjclass[2000]{60J05, 60F17, 35Q91, 35Q20, 60J10, 60J20, 82B31, 82B40} 
\date{\today}

\maketitle

\begin{abstract}
We discuss various limits of a simple random exchange model that can be used for the distribution of wealth. We start from a discrete state space - discrete time version of this model and, under suitable scaling, we show its functional convergence to a continuous space - discrete time model. Then, we show a thermodynamic limit of the empirical distribution to the solution of a kinetic equation of Boltzmann type. We solve this equation and we show that the solutions coincide with the appropriate limits of the invariant measure for the Markov chain. In this way we complete Boltzmann's program of deriving kinetic equations from random dynamics for this simple model. Three families of invariant measures for the mean field limit are  discovered and we show that only two of those families can be obtained as limits of the discrete system and the third is extraneous. Finally, we cast our results in the framework of integer partitions and strengthen some results already available in the literature.
\end{abstract}

\section{Introduction}

This study was originally motivated by a new approach to macroeconomics modelling based on (1) continuous-time Markov chains to model stochastic dynamics interactions among agents and (2) combinations of stochastic processes and non-classical combinatorial analysis, called combinatorial stochastic processes. Such an approach was extensively presented in \cite{Aoki07}. Those authors argue that, in case (1), the master equation describes how states of the models evolve stochastically in time and, in case (2), combinatorial stochastic processes are applied to describe the random formation of clusters of agents as well as the distribution of cluster sizes.
Mathematically, the two approaches are so strictly related that it is not necessary to distinguish between them. This point was already implicitly made in Chapter 10 of \cite{garibaldiscalas}. Moreover, both approaches are related to kinetic equations of Boltzmann type used in statistical physics \cite{ParTosBook}.  

We previously worked on the class of Markov-chain models described below in \cite{DGS} where we focused on the existence and uniqueness of the invariant measures and on the stability of the Markov chains. Some results in this article can be found in the expository chapter \cite{DGMS}, written with an eye for economists and with all the proofs omitted. 

We explore the connection between combinatorial stochastic processes and kinetic equations of Boltzmann type via functional limit theorems of properly scaled processes in the spirit of \cite{Billingsley,ethier2009markov,Jacod}.

In this article we study a simple discrete model for wealth dynamics using a coagulation - fragmentation process. This is the same as the one in \cite{DGS, DGMS}. 

The \emph{discrete space, discrete time (DS-DT)} model is a Markov Chain on the integers partitions of $n$ that have size $N$. 
In other words, the state space is comprised of all non-negative integer vectors $ {\bf x}^{n, N} =  (x_1, \ldots x_N) \in \Z_+^N$ so that $\sum_{i=1}^N x_i  = n$. 
The $x_i $'s represent the wealth of the $i$-th individual and the superscripts are there to remind us of the total wealth and number fo agents. 
We denote the state space by $S^{(n)}_{N-1} = n \Delta_{N-1} \cap \Z^{N}$,  where 
\be \label{eq:D_N}
\Delta_{N-1} = \Bigg\{ \mathbf{x} = (x_1, \ldots, x_N): \, x_i \geq 0 \, \text{ for all } i = 1, \ldots, N \text{ and } \sum_{i=1}^N x_i = 1 \Bigg \},
\ee
is the $N$-dimensional unit simplex.

At every discrete time step, we choose an ordered pair of indices from $1$ to $N$ uniformly at random (say $(i,j)$) and add the individual wealths $x_i + x_j$ of the the agents. After that, the first chosen agent $i$ receives a uniform portion of the total wealth between $0$ and $\max\{x_i + x_j - 1, 0\}$ and the rest goes to the second agent $j$. 
Let ${\bf X}_t^{n, N}$ denote the wealth distribution at time $t$. The transition probabilities for this chain are given by 
\begin{align}
\label{transition}
\mathbb{P}&\{\mathbf{X}^{n, N}_{t+1}= \mathbf{x}' | \mathbf{X}^{n,N}_t=\mathbf{x}\} \notag\\
&=\sum_{(i,j): i \neq j}\left\{  \frac{1}{N} \frac{1}{N-1}\Bigg(\frac{\mathbbm1\{ x_i + x_j \ge 1, x'_j\ge 1\}}{x_i+x_j} + \mathbbm1\{ x_i + x_j = 0\} \Bigg)
\delta_{x_i+x_j,x'_i+x'_j} \prod_{k \neq i,j} \delta_{x_k,x_k'} \right\}.
\end{align}
As it turns out, the transition matrix for the chain is doubly stochastic, therefore the invariant distribution is uniform on $S^{(n)}_{N-1}$ which is also obtained as $t \to \infty$ because of irreducibility and aperiodicity. 

After studying the discrete chain, it would be more realistic to allow the total wealth $n$ to increase, but in general that would only alter the state space. However, 
there is way to converge to a \emph{continuous space, discrete time (CS-DT) model}, if we alter the discrete model slightly. In particular, instead of looking at the distribution of wealth, we look at the distribution of the {\bf proportion } of wealth, namely the process ${\bf Y}^{n, N} = n^{-1} {\bf X}^{n, N}$ 
which is a rescaling of the original discrete process by the total wealth. 
The state space for the  ${\bf Y}^{n, N}$ process is the meshed simplex
\be \label{eq:D_Nn}
\Delta_{N-1}(n) = \Big\{ (q_1, \ldots, q_N): 0 \le q_i \le 1, \sum_{i=1}^N q_i = 1, nq_i \in \N_0 \Big\} \subset \Delta_{N-1}.
\ee
Then in \cite{DGS}, it was shown (Proposition 3) that as $n \to \infty$ one had the weak convergence of one-dimensional marginals  
\be
{\bf Y}_t^{n, N} \Longrightarrow {\bf X}_t^{\infty, N}\, \text{ as } n \to \infty,
\ee
under the mild assumption that the initial distributions of ${\bf Y}^{n, N}$, $\mu_0^{n, N}$ converge weakly to some distribution $\mu_0^{\infty, N}$ on $\Delta_{N-1}$.
Process ${\bf X}_t^{\infty, N}$ is identified as a continuous space, discrete time Markov chain on $\Delta_{N-1}$. At each discrete time step $t$, an ordered pair of agents, say $(i,j)$ is selected uniformly at random, with total proportion of wealth $x_i+x_j$. Then an independent uniform random variable $u_{t, (i,j)} \sim\text{Unif} [0,1]$ is drawn and the new proportion of wealth for agent $i$ is $u_{t, (i,j)} (x_i + x_j)$  while for agent $j$ is $(1 - u_{t, (i,j)})(x_i + x_j)$. 
Note that the wealth of agents are exchangeable random variables; while the description above needs ordered pairs of agents, it has no bearing on the distribution of the eventual wealth, as both $u_{t,(i,j)}$ and $1 - u_{t,(i,j)}$  are uniformly distributed on $[0,1]$.

For the  CS-DT chain ${\bf X}_t^{\infty, N}$, it was further shown that the invariant distribution of wealth proportions as $t \to \infty$ is uniform on $\Delta_{N-1}$. 

Here, we go a few steps further. First, we show the process level convergence 
\be
{\bf Y}^{n, N} \Longrightarrow {\bf X}^{\infty, N}\, \text{ as } n \to \infty,
\ee
by showing convergence of the finite dimensional marginals of the process. 
Then, using the Poissonization trick \cite{Pollard}, we will change time and consider a continuous-time version of our continuous-space Markov chain. In an other appropriate scaling limit, this will lead to one-dimensional kinetic equations of Boltzmann type as studied e.g. in \cite{Bassetti}. Stochastic mean-field dynamics for interacting particle systems are well-studied; for example see  \cite{daipra2017} for models where components are exchangeable, as in our model here. 


%

\subsection{Kinetic equations for wealth models}
A one-dimensional caricature of the three-dimensional Boltzmann equation for Maxwell molecules is the Kac model. While simpler, it retains key properties of the original Boltzmann equation, such as energy conservation in binary collisions. The Kac equation has been deeply analysed using Fourier analysis techniques, e.g. in \cite{bobylev1988theory, bobylev2000some}. This model also allows for a rigorous passage from the kinetic model with binary interactions to a Fokker–Planck equation in the grazing collisions limit \cite{toscani1998grazing, villani2002review}.

In the last two decades, the mechanism of the binary interaction, originally developed for the Boltzmann equation, has been fruitfully adapted to describe collective dynamics in many-agent socio-economic systems. The basic idea is to describe the behaviour of a sufficiently large number of interacting agents in the socio-economic system by pairwise, microscopic interactions, similar to the physical models of rarefied gas dynamics, where molecules collide inside a container. One can then study the long-time dynamics of the system and observe the formation of macroscopic distributions, depending on the details of the microscopic interactions. This approach has been successfully followed to model wealth distribution in simple market economies \cite{cordier2005kinetic, during2007hydrodynamics, during2008kinetic, during2008international}, wealth distribution under taxation \cite{toscani2009wealth, during2018kinetic}, opinion formation \cite{toscani2006kinetic, during2009boltzmann, during2015opinion}, asset pricing \cite{during2016kinetic}, continuous models for ratings \cite{jabin2015continuous, during2019boltzmann}, and others.

\subsection{Kinetic equations as limits of discrete particle models}

The derivation of kinetic equations from discrete particle models is classical in kinetic theory. This is generally a hard problem that involves proving that `propagation of chaos' holds for the system. This corresponds to showing that the particles become statistically independent when their number grows large. Typically, proving propagation of chaos allows to close the BBGKY hierarchy, i.e., the hierarchy of equations giving the evolution of the marginal distributions associated to the system \cite{cercignani2013mathematical,spohn1984boltzmann}. In the case of the classical Boltzmann equation, which describes hard-sphere collision dynamics, the kinetic limit was shown in \cite{lanford1975time}, though there is still a proof missing for long times \cite{gallagher2012newton}. 

In this work we will use a probabilistic approach in order to obtain the kinetic equation of the system under consideration. On this account, Sznitmann \cite{sznitman1991topics} showed the kinetic limit for McKean-Vlasov systems of Stochastic Differential equations using a coupling argument. This argument has been further extended recently in \cite{diez2019propagation} to a piece-wise deterministic Markov process. Previous works also investigate the speed of convergence to the kinetic equation in terms of the number of particles. 

For our results, we must use a different approach, since we consider a pure jump process. Particularly, the methodology used is based on computing the limit of the martingale formulation associated to the jump (Markov) process. The methodology used here has been applied with great success to the investigation of coagulation models and the Smoluchowski equation in \cite{norris1999smoluchowski,norris2000cluster} and later to a system of instantaneous coagulation-fragmentation processes in \cite{merino2016isotropic}. 

\subsection{Content and structure} In Section \ref{sec:models}, we introduce the three connected models of the evolution of wealth, and present our results. The first one is an alternative formulation of the discrete model (discrete space, discrete time) with conserved wealth. The state space of the process is a discrete finite dimensional simplex. The dimension is the number of agents, and at each time step two agents interact (or collide). The Markovian evolution of the process is that of a discrete coagulation-fragmentation process. 

The second model is obtained as a scaling spatial limit of the first one and is effectively the continuous space, discrete time analogue. Section \ref{sec:procon} is dedicated to show process level convergence from the discrete to the continuous space model. Finally, the third model is the mean-field continuous limit for the empirical distribution of wealth. Agents are viewed as particles with binary interactions. In Section \ref{sec:sara2}, by assuming the coagulation-fragmentation process jumps at the times of a Poisson process and letting the number of agents tend to infinity while appropriately scaling time, we obtain the relevant kinetic equations. 

Section \ref{sec:sara} is concerned with invariant distributions for the kinetic equation. While we find at least three potential invariant measures for the limiting empirical wealth (a delta, an exponential and a family of truncated exponential distributions) we show that from the particle system description only two of these are acceptable limits (the delta and the exponential). This highlights the power of the probabilistic approach, as a purely analytical one would not be able to {\em a priori} exclude that family. Similar laws of large numbers for empirical measures of particle systems can be found for a huge class of processes in the literature, e.g.\cite{grosskinsky2019}.

Finally, Section \ref{sec:partitions} is an application of this theory when we view the process not as a wealth evolution, but the evolution of a Markov chain on integer partitions. As a by-product we recover some theorems of \cite{vershik2003asymptotics}. Results of Section \ref{sec:partitions} are not mentioned earlier in the paper, and the interested reader can directly start reading that section. 
 
To make the paper as self-consistent as possible, we have included an appendix on functional limit theorems for stochastic processes. 

\subsection{Acknowledgements:} We would like to thank Christina Goldschmidt and Stefan Grosskinsky for valuable and interesting discussions and for suggestion of related references.
 
\section{The models and results}
\label{sec:models}
 
We briefly describe the various models we are using, and collect the main results for an organised reference.
 
We consider $N$ agents (originally $N$ is fixed)  and wealth $W_N$ (originally fixed to be and integer denoted by $n$). 

\subsection{Equivalent construction of the DS-DT process.}

For any $n \in \N$ the process ${\bf Y}^{(n)}$ is defined on $\Delta_{N-1}(n)$ given by $\eqref{eq:D_Nn}$, and we emphasise that  for every $n$, 
$\Delta_{N-1}(n) \subset \Delta_{N-1}$, given by $\eqref{eq:D_N}$. $\Delta_{N-1}(n)$ is treated as the meshed simplex $\Delta_{N-1}$;
the mesh size is $n^{-1}$, which is precisely the reciprocal of the total wealth $W_n = n$.  

Let $\mathcal{P}^n$ denote the law of the process ${\bf Y}^{(n)} = ({\bf Y}^{(n)}_0, {\bf Y}^{(n)}_1, \ldots,  {\bf Y}^{(n)}_k, \ldots ) \in (\Delta_{N-1}(n))^{\N_0} \subset (\Delta_{N-1})^{\N_0}$.
The measure for $k+1$-th dimensional marginal $({\bf Y}^{(n)}_0, {\bf Y}^{(n)}_1, \ldots,  {\bf Y}^{(n)}_k)$ is denoted by 
\be\label{eq:induced}
\mathcal{P}^n_k\{ \cdot\} = \mathcal{P}^n\big\{ ({\bf Y}^{(n)}_0, {\bf Y}^{(n)}_1, \ldots,  {\bf Y}^{(n)}_k) \in \cdot \big\}.
\ee
Similarly, denote by $\mathcal{P}^\infty$ and $\mathcal{P}^\infty_k$ the corresponding quantities for ${\bf X}^{\infty}$.  The law of ${\bf Y}_0^{(n)}, {\bf X}_0^{(\infty)}$ are denoted by $\mu_0^{(n)} = \mathcal P^{(n)}_0$ and  $\mu_0^{(\infty)} = \mathcal P^{(\infty)}_0$ respectively.

Starting from an initial distribution $\mu_0^{(n)}$ we construct the process ${\bf Y}^{(n)}$ using an i.i.d.\ sequence of uniform random variables  
\be \label{eq:newuni}
U^{(n)}_{i, j}(k) \sim \text{Unif}[0,1], \qquad 1 \le i , j \le N, i \neq j, k \in \N_0, n \in \N.   
\ee
%
These random variables from \eqref{eq:newuni} suffice to construct the whole process. The variable $k$ plays the role of time index, and $(i,j)$ is the ordered pair of agents that are selected. We assume -and use without a particular mention- that random variables \eqref{eq:newuni} are independent of the initial distribution $\mu_0^{(n)}$. 

For any $x \in \R_+$ we define 
\[
[ x]_n =  \frac{a}{n}, \quad \text{  so that  } \frac{a}{n} \le x <  \frac{a+ 1}{n},  \quad a \in \N_0,
\]
and use this symbol for notational convenience when we define the evolution of the process directly on $\Delta_{N-1}(n)$.

 Let 
${\bf Y}_k^{(n)} = (y_1(k), \ldots y_N(k)) \in \Delta_{N-1}(n)$ be the vector of discrete wealths, normalised so that the total wealth is $1$. Then, if indices $i.j$ were chosen to interact at time step $k$, the total wealth at time $k+1$ would become
\begin{align*}
{\bf Y}_{k+1}^{(n)} &= (y_1(k), \ldots, \underbrace{[U_{i,j}^{(n)}(k)(y_i(k) + y_j(k)) ]_n}_{y_{i}(k+1)}, \ldots,  \underbrace{y_i(k) + y_j(k) -y_{i}(k+1)}_{y_{j}(k+1)}, \ldots, y_N(k)) \\
&=g_{i,j}(\y_k, U_{i,j}^{(n)}(k)).
\end{align*}
Check to see that the coordinate 
$[U_{i,j}^{(n)}(k)(y_i(k) + y_j(k)) ]_n$ is uniformly distributed on the set $\{ 0, n^{-1}, \ldots, (y_i(k) + y_j(k) - n^{-1})\vee 0 \} $, and therefore this procedure gives the same process as described in \cite{DGS}. The function $g_{i,j}$ is a measurable function that depends on the value of the current state and the new uniform random variable, and the last display acts as the definition of $g_{i,j}$. 

We prove the following theorem, which guarantees process-level convergence.

\begin{theorem}\label{thm:processlevel}
Assume the weak convergence of measures 
\be
\mu_0^{(n)} \Longrightarrow \mu_0^{(\infty)}, \quad \text{ as } n \to \infty.
\ee
Furthermore, assume the weak convergence (as $n \to \infty$) of the i.i.d.\ sequence 
\be
 \{ U^{(n)}_{i, j}(k) \}_{i,j,k} \Longrightarrow \{ U^{(\infty)}_{i, j}(k) \}_{i,j,k},
\ee
so that the limiting sequence $\{ U^{(\infty)}_{i, j}(k) \}_{i,j,k}$ is a sequence of i.i.d.\ uniform $[0,1]$ random variables that are also independent from  $\mu_0^{(\infty)}$.

Then 
\[
\mathcal{P}^n \Longrightarrow \mathcal{P}^{\infty}, \quad \text{ as } n \to \infty.
\]
\end{theorem}

The theorem gives that the order in which we take limits in the diagram of Fig. 1 is immaterial and the diagram is commutative.
\begin{figure}[h]
\begin{center}
\begin{tikzpicture}[>=latex, scale=1.5]
\draw(-1,5)node{DS-DT, ${\bf Y}^{n , N} = n^{-1}{\bf X}^{n, N} \in \Delta_{N-1}(n)$} ;
\draw[->] (1.3,5)--(4,5);
\draw(2.5,5.25)node{$ n \to \infty $} ;
\draw(5.5,5)node{CS-DT, ${\bf X}^{\infty, N} \in \Delta_{N-1}$};
\draw[->](-1,4.5)--(-1,2);
\draw(-0.25,3.25)node{$t \to \infty$};
\draw[->](5.5,4.5)--(5.5,2);
\draw(6.25,3.25)node{$t \to \infty$};
\draw(-1,1.5)node{DS-DT, $ \mu^{n, N}_{\infty} \sim \text{Unif}(\Delta_{N-1}(n))$} ;
\draw[->] (1,1.5)--(3.8,1.5);
\draw(2.5,1.75)node{$ n \to \infty $} ;
\draw(5.5,1.5)node{CS-DT, $\mu^{\infty, N}_{\infty} \sim \text{Unif}(\Delta_{N-1})$};
\end{tikzpicture}
\caption{ Commutative diagram demonstrating the various limiting measures, depending on the order limits are taken, when the total wealth remains constant. Measures  $\mu^{n, N}_{\infty}$  and  $\mu^{\infty, N}_{\infty}$ denote the invariant distributions for the two Markov chains respectively.}
\end{center}
\label{fig:D1}
\end{figure}

This will be proven in Section \ref{sec:procon}. Horizontal arrows in the diagram of figure 1 denote weak convergence, but the top one can be upgraded to almost sure convergence if we are concerned with finite sample paths. 

Moreover, we will investigate the mean field limit of the CS-DT process, as $N \to \infty$. In order to do this using kinetic theory, it is useful to switch to a continuous time Markov chain, where jump times coincide with those of a rate 1 Poisson process, which is why it is called a ``Poissonisation trick". It is standard to argue that the long time behaviour of the discrete time process is the same as that of the Poissonised one when $N$ is fixed, irrespective of the rate of the Poisson process. The finite time distribution of the proportions of wealth for the CS-CT Poissonised process, which we momentarily denote by ${\bf X}_t^{\text{Pois}}$, can also be rigorously obtained by standard conditioning on the number of Poisson events up to time $t$, using the following equation
\be
\P\{ {\bf X}_t^{\text{Pois}} \in A \} = \sum_{\ell = 0}^{\infty} \mathbf P\{ {\bf X}^{\infty, N}_\ell \in A \} P\{ N_t = \ell \} = \sum_{\ell = 0}^{\infty} \mathbf P\{ {\bf X}^{\infty, N}_\ell \in A \} \frac{e^{-t/N}t^{\ell}}{\ell ! \, N^{\ell}}.
\ee
 $N_t$ is the background Poisson process with rate $1/N$ and $A$ is any Borel subset of the simplex. We omit the argument that the limiting  distribution is still uniform on the simplex.

\subsection{Martingale formulation for the CS-DT model.}
\label{sub:mgf}

In general, it is not necessary to restrict to a case  where the total wealth is 1 for all $N$, the same models can be studied when the total wealth is a function of $N$; here we do so for the kinetic model. Let us first introduce some notation. The total wealth in a system of $N$ agents is denoted by a value $W_N \in \R_+$ (which we also allow to be 0).
The state of the process at time $t$ is a vector of non-negative real numbers 
\[
{\bf X}_t^{N} = (X^{1,N}_t, \ldots, X^{N,N}_t )
\] 
with state space 
\[
\Delta_{W_N}:=\Big\{(x_1,\hdots, x_N) : x_i \ge 0 \text{  for all $1 \le i \le N$  } \text{ and  }  \sum_{i=1}^N x_i=W_N\Big\}.
\]
The dynamics on $\Delta_{W_N}$ are given by binary interactions, where an ordered pair of two agents $(i,j)$ is chosen uniformly at random. The interactions are assumed to happen at constant rate $1/N$, at the events of a background Poisson process. After the interaction, the wealth of the pair $(X^{i,N}, X^{j,N})$ is changed to $((X^{i,N})', (X^{j,N})')$ with
\beqar
(X^{i,N})' &=& r(X^{i,N}+X^{j,N}),\\
(X^{j,N})' &=& (1-r) (X^{i,N}+X^{j,N}),
\eeqar
where $r$ is a random variable with uniform law on $[0,1]$ that is drawn at time $t$, independently of the past of the chain. Interactions preserve the total mass,
\be \label{eq:total_wealth}
W_N := \sum_{i=1}^N X^{i,N},
\ee
and, therefore, the dynamics take place on $\Delta_{W_N}$.
We will consider two cases:
\begin{itemize}
	\item[(i)] Absolute wealth: $X^{i,N}$ represents the wealth of agent $i$ and $W_N$ represents the total wealth of the system;
	\item[(ii)] Relative wealth: in this case $X^{i,N}$ represents the proportion of wealth of agent $i$ and $W_N=1$ for all $N$.
\end{itemize}

We are interested in studying the case when the number of agents grows large, i.e., $N\to \infty$. The first thing to observe is that agents are \emph{exchangeable} by virtue of the non-preferential dynamics. Questions of interest also reflect that, in the sense that we want to know how much wealth the richest agent has, rather than who is the richest agent, since they all have the same probability of being the most rich. For this reason, we will focus our study on the empirical distribution
\be \label{eq:emp}
\mu^N_t(x) = \frac{1}{N}\sum_{i=1}^N\delta_{X^{i,N}_t}(x).
\ee
The empirical distribution $\mu^N_t$ is a random probability measure on $\R_+$ that depends on the realisation of the Markov chain. For any interval $[a, b]$, 
\[
\mu^N_t([a, b]) =  \frac{1}{N}\sum_{i=1}^N\delta_{X^{i,N}_t}[a, b] = \frac{1}{N}\sum_{i=1}^N \mathbbm1\{ a\le X^{i,N}_t \le  b \} =  \frac{\text{card}\{ i: \text{agent $i$'s wealth } \in [a, b]\}}{N}.
\] 
In general, for any measure $\mu$ on $\R_+$, and any $\mu$-measurable function $g$, we define 
the brackets $\la \cdot, \cdot \ra$ by
\be \label{eq:def_brackets}
\la g, \mu\ra := \int_{\R_+} g(x)\mu(dx).
\ee
When $\mu$ is a probability measure, the bracket notation is just another way to denote the expected value $\E_{\mu}(g)$.
With this definition, when the measure is  $N\mu^N_t(x_0)$ for some fixed $x_0$, the bracket  
 $\la 1, N\mu_t(x_0)\ra$ gives the number of agents with wealth precisely $x_0$ at time $t$. Equivalently, keep the empirical measure as $\mu^N_t(x)$ and set $g(x) = N\sum_{i=1}^N\mathbbm1\{ X^{i, N}_t = x_0 \}$ in order to obtain the same interpretation.

The total wealth in the system represented by  $W_N$ at time 0 as in Eq. \eqref{eq:total_wealth}, and we can write this fact in terms of the empirical distribution as 
\be \label{eq:total-wealth} 
W_N = N\la x, \mu^N_0 \ra.
\ee
The total wealth at time $t$ is given by  $N\la x, \mu^N_t \ra$ and it remains fixed  for all $t \ge 0$ if we assume a conserved total wealth. 
Notice that if $W_N/N\to m$ as $N\to \infty$ then we also have that 
\be \label{eq:limit1}
\lim_{N \to \infty} \la x, \mu^{N}_0\ra = m.
\ee
If $\mu^{N}_0 \Longrightarrow \mu_0$ weakly for some probability measure $\mu_0$ and $m=0$, Eq. \eqref{eq:limit1} would imply that $\mu^{N}_0(x) \Longrightarrow \delta_0(x)$, as the measure has no support in the negative reals.
 
For a fixed $t$, the empirical measure $\mu_t^N$ is an element of the space of probability measures $\mathcal{M}_1$ on $\R_+$ and it only changes whenever an interaction event occurs. It is a function of the Markov chain ${\bf X}_t^N$ and it is also a Markov chain.

In order to describe its generator $\mathcal{G}$, we define the measure $\mu^{(x,y,r),N}$ after an interaction between an agent of wealth $x$ (chosen first) and one of wealth $y$ (chosen second) to be
\[ \mu^{(x,y,r),N}= \mu^N -\frac{1}{N}\delta_{x}-
\frac{1}{N}\delta_{y}+ \frac{1}{N}\delta_{r(x+y)}+ \frac{1}{N}\delta_{(1-r)(x+y)}.
\]
Finally, we define the pair-measure $\mu^{(2,N)}$ on rectangles that generate the Borel $\sigma$-algebra 
$\mathcal{B}(\R\times\R)$
to be 
\be \label{eq:pairmeasure} 
\mu^{(2,N)}(A\times B)= \mu^N(A)\mu^N(B)-\frac{1}{N}\mu^N(A\cap B), \qquad A, B\in \mathcal{B}(\R).
\ee
This is a natural choice of the pair measure, as it is a simplified version of the joint empirical measure for a pair of variables. Note that it is not a probability measure, but this does not matter, as we will only use it as $N\to \infty$. For more clarification and details see Remark \ref{rem:1} at the end of the section. 

The generator for the evolution of $\mu_t^N$, considering an interaction rate of $1/N$, is given by
\be \label{eq:generator}
\mathcal{G}F(\mu^N) = \int_0^1\doubleint \{ F(\mu^{(x,y,r),N}) - F(\mu^N) \}{\mathbbm{1}_{\{x+y\leq W_N\}}}N\mu^{(2,N)}(dx,dy)\, dr.
\ee
In the equation above, function $F$ belongs to $C_b(\mathcal{M}_1)$, i.e., bounded measurable functions on the space of probability measures $\mathcal{M}_1$. We impose the term
	$\mathbbm1_{\{x+y \leq W_N\}}$
	in the generator to ensure that the two masses created after the jump fulfil $r(x+y)\leq W_N$ and $(1-r)(x+y)\leq W_N$. 

\begin{remark}	
In this manner, we could consider that $\mu^N_t \in  \mathcal{P}([0, W_N])$. However, to avoid having a functional space depending on the value of $N$, we will just consider that $\mu^{N}_t \in \mathcal{P}(\R_+)$. Notice that the generator can be also interpreted as representing a $N$-particle system with values in $\R_+$ where only pair of values interact as long as their sum is below $W_N$.
\end{remark}

Given the generator in \eqref{eq:generator}, we have that the quantity  $M^F_t$ defined by 
\be \label{eq:martingale}
M^F_t = F(\mu^N_t) - F(\mu^N_0)-\int^t_0 \mathcal{G}F(\mu^N_s) \, ds
\ee
is a martingale  \cite[Appendix]{kipnis2013scaling}, for any $F \in C_b(\mathcal M_1)$.
In particular, for any function $g\in C_b(\R_+)$ (measurable bounded functions in $\R_+$), we define $F_g \in  C_b(\mathcal M_1)$ by  $F_g(\mu)=\la g, \mu \ra := \int g(x)\mu(dx)$.
Expression \eqref{eq:martingale} can now be re-written as
\be \label{eq:martingale_general}
M^{g,N}_t = \la g, \mu^N_t \ra - \la g, \mu^N_0\ra -\int^t_0 \la g, Q^{(N)}(\mu^N_s) \ra\, ds,
\ee
where we are denoting  $\mathcal{G}(\la g, \mu^N \ra)$ by
\begin{align}
 \mathcal{G}(\la g, \mu^N \ra) &=  \int_0^1\int_{\R_+}\int_{\R_+} \Big(g(r(x+y)) + 
 g((1-r)(x+y)) -g(x) - g(y) \Big){\mathbbm{1}_{\{x+y\leq W_N\}}} \mu^{(2,N)}(dx,dy)\, dr  \notag\\ \label{eq:def_QN}
 &= \la g, Q^{(N)}(\mu^N)\ra.
\end{align}
The last line in fact allows us to define $Q^{(N)}(\mu)$ implicitly via its brackets with bounded continuous functions $g$. 

In the following sections we will see that $\mu^N_t$ converges in probability as $N\to \infty$ to a measure $\mu$ which is solution of the following kinetic equation in weak form:
\be \label{eq:kinetic_equation}
\mu_t = \mu_0 + \int^t_0  Q(\mu_s) \ ds,
\ee
where the operator $Q$ is defined as follows: for any $g\in C_b(\R_+)$ 
\begin{align} 
\label{eq:operator_Q}
\la g, Q(\mu)\ra &= \int_{[0,1]}\int_{\R_+}\int_{\R_+}\lp g(r(x+y))+ g((1-r)(x+y))- g(x)-g(y) \rp {\mathbbm{1}_{\{x+y\leq w_0\}}}\mu(dx)\mu(dy)\, dr,
\end{align} 
with $w_0=\lim_{N\to\infty}W_N.$ We will also investigate the limit $t \to \infty$ and obtain different families of limiting invariant measures, in the process verifying the following commutative diagram of Fig. 2 in the simple case of fixed wealth $W_N = c$ for all $N$.
\begin{figure}[h]\label{fig:2}
\begin{center}
\begin{tikzpicture}[>=latex, scale=1.3]
\draw(5.5,5)node{CS-DT, $\{\mu_t^N\}_{t \ge 0}$};
\draw[->] (7,5)--(9,5);
\draw(8,5.25)node{$ N \to \infty $} ;
\draw(8,4.75)node{M-F, Poissonisation} ;
\draw(10,5)node{$\{ \mu_t \}_{t \ge 0} \in \R_+$};
\draw[->](5.5,4.5)--(5.5,2);
\draw(6.25,3.25)node{$t \to \infty$};
\draw[->](10,4.5)--(10,2);
\draw(10.75,3.25)node{$t \to \infty$};
\draw(5.5,1.5)node{CS-DT, $\mu^{N}_{\infty} $};
\draw[->] (6.5,1.5)--(9.5,1.5);
\draw(8,1.75)node{$ N \to \infty $}; 
\draw(8,1.2)node{M-F, Poissonisation}; 
\draw(10.3,1.5)node{$\mu_{\infty} \sim \delta_0$};
\end{tikzpicture}
\caption{ Commutative diagram demonstrating the various limiting measures, depending on the order limits are taken, when the total wealth remains constant. There are two parameters that scale;  the number of agents $N$ and the time $t$. Time is discrete for the left down-arrow, but continuous in the right down-arrow. There is an intermediate step missing from the diagram in which discrete time events are changed with time events arising from a Poisson process of rate $1/N$ which simultaneously scales with $N$. That is called the Poissonisation step, and when the mean-field limits (M-F) are taken, the rate of the Poisson process also scales with $N$. }
\end{center}
\end{figure}

In Fig. 2, the left down-arrow was obtained in \cite{DGS}. The lower horizontal arrow is  obtained in the present article in Proposition \ref{lem:equilibria}, and the remaining arrows in Sections \ref{sec:sara} and \ref{sec:sara2}.

\begin{remark}\label{rem:1}
Equation \eqref{eq:pairmeasure}
 is a natural choice for the pair measure, as the following calculation demonstrates. We begin from the joint empirical measure 
 \[
 \nu^N_t(x,y)= \frac{1}{N(N-1)}\sum_{(i,j): i \neq j} \delta_{ (X^{i, N}_t,   X^{j, N}_t)}(x,y). 
 \]
 On general product events $A \times B $ the measure can be computed as 
 \begin{align*}
 \nu^N_t(A \times B) &= \frac{1}{N(N-1)}\sum_{(i,j): i \neq j} \mathbbm1\{ X^{i, N}_t \in A,  X^{j, N}_t \in B\} \\&= \frac{1}{N(N-1)}\sum_{(i,j): i \neq j} \mathbbm1\{ X^{i, N}_t \in A\}\mathbbm1\{ X^{j, N}_t \in B\} \\
 &= \frac{1}{N(N-1)} \sum_{i} \mathbbm1\{ X^{i, N}_t \in A\} \sum_{j}\mathbbm1\{ X^{j, N}_t \in B\} - \frac{1}{N(N-1)} \sum_{i} \mathbbm1\{ X^{i, N}_t \in A\cap B\}\\
 &= \frac{N}{N-1} \mu^N_t(A) \mu^N_t(B) - \frac{1}{N-1}\mu_t^N( A\cap B) = \frac{N}{N-1} \mu^{(2, N)}_t(A \times B).
 \end{align*}
 As $N \to \infty$ the prefactor  $N/(N-1) \to 1$ and the limiting measure has the same asymptotic properties. We choose to use the simplest form \eqref{eq:pairmeasure} without loss of generality.
\end{remark}

\begin{theorem}[Mean-field limit]
\label{th:hydrodynamic_limit}
Suppose that $W_N$ is a non-decreasing sequence converging to $w_0\in (0, \infty]$ as $N\to\infty$.
Suppose that for a given measure $\mu_0$ it holds that
\be \label{eq:bound_initial_data}
\la x, \mu^N_0\ra \leq \la x, \mu_0\ra<\infty,
\ee
and that as $N\rightarrow \infty$
\be \label{eq:limit_initial_data}
\mu^N_0 \Longrightarrow \mu_0 \qquad \mbox{weakly,  as $N\to \infty$}.
\ee
Then the sequence of random measures $(\mu^N_t)_{t\geq 0}$ converges in probability in $D([0,\infty); \mathcal{M}_1(\R_+))$, as $N\rightarrow \infty$. The limit $(\mu_t)_{t\geq 0}$ is continuous in $t$ and it satisfies the kinetic equation  \eqref{eq:kinetic_equation}. In particular, for all $g\in C_b(\R_+)$ the following limits hold in probability, for any time $t$
\begin{enumerate}
\item[(A)] $ \displaystyle \lim_{N \to \infty} \sup_{s\leq t}\, \la g, \mu^N_s -\mu_s\ra \stackrel{\P}{=}  0, $ \vspace{0.3cm}
\item[(B)] $ \displaystyle \lim_{N \to \infty} \sup_{0\leq s\leq t }|M_s^{g,N}| \stackrel{\P}{=}    0, \,\, $ \vspace{0.3cm}
\item[(C)] $\displaystyle  \lim_{N \to \infty} \int^t_0 \la g, Q^{(N)}(\mu^N_s)\ra\, ds \stackrel{\P}{=}    \int^t_0 \la g, Q(\mu_s) \ra\, ds. $
\end{enumerate}
As a consequence, equation \eqref{eq:kinetic_equation} is obtained as the limit in probability of \eqref{eq:martingale_general} as $N\rightarrow \infty$. 
\end{theorem}


Some observations from Theorem \ref{th:hydrodynamic_limit} follow. From equation \eqref{eq:total-wealth} we have that $N \la x, \mu^N_0\ra = W_N$. 
If we now assume that $\displaystyle \lim_{N \to \infty} N^{-1} W_N = m \in (0, \infty)$ then we see that $W_N$ grows linearly in $N$ and condition \eqref{eq:bound_initial_data}
implies 
\[ 
m \leq \la x, \mu_0\ra.
\]	

Now if  $W_N$ grows superlinearly, i.e. $ \displaystyle \lim_{N \to \infty} N^{-1} W_N = \infty,$
then condition \eqref{eq:bound_initial_data} in Theorem \ref{th:hydrodynamic_limit} is violated and the theorem does not necessarily hold.  

Finally, if either $\displaystyle\lim_{N \to \infty}W_N = w_0$ for some absolute constant $w_0$ or  $W_N\to \infty$ as $N\to\infty$, but $\displaystyle \lim_{N \to \infty} N^{-1} W_N = 0,$ we can actually study the asymptotic behaviour ($N \to \infty$) of the measures $\mu_t^N$ and show that the limiting measure is a $\delta$ mass as $N \to \infty$. This is discussed in Section \ref{sec:sara2}.

\subsection{Invariant measures for the mean field limit}

In general, a measure $\tilde \mu$ is invariant (or stationary) for \eqref{eq:kinetic_equation} if and only if when $\mu_0 =\tilde  \mu$ then we have that $\mu_t = \tilde \mu$ for all $t >0$. 

One way to obtain invariant measures is to actually make some educated ansatz for $\mu_0$ and show that it remains unchanged under the kinetic equation \eqref{eq:kinetic_equation}. It is immediate to check that for any value of  $w_0$ (bounded or unbounded), the measure 
	\be \tilde \mu (x) = \delta_0(x)\ee
is invariant for \eqref{eq:kinetic_equation}. 

A more natural way to find invariant measures originates from the Markov chain perspective, where (limiting) \emph{equlibrium} measures $\bar \mu$ are obtained by taking the limit (in the appropriate weak sense) of the measures $\mu_t$ as $t \to \infty$, i.e.
\[
\bar \mu = \lim_{t\to \infty} \mu_t,
\]
if such a limit exists, and then the measure $\bar \mu$ will be invariant. A sequence of measures however may have many limit points; it is always an important and difficult task to decide whether those limit that are obtained include all possible equilibria for the system. Moreover, the limiting measure(s) will depend on the initial measure $\mu_0$ and other parameters of the evolution.

In this subsection, we discuss several invariant measures that can be obtained as equilibria. 
We begin with the case where $W_N$ grows sublinearly and we show that under Theorem \ref{th:hydrodynamic_limit}, $\delta_0$ is the only possible candidate for invariant equilibrium measure. Proposition \ref{cor:delta} indeed asserts that result, under the assumptions of Theorem \ref{th:hydrodynamic_limit}, and Proposition \ref{lem:equilibria} argues that the assumptions of Theorem \ref{th:hydrodynamic_limit} hold when the total wealth $w_0 = 1$ and we start from a uniform density on the simplex. Together, these propositions verify the commutativity of the diagram in Figure \ref{fig:2}.

\begin{proposition}[Sub-linear growth for  $W_N$]
\label{cor:delta}
Suppose the same assumptions on the initial data as in Theorem \ref{th:hydrodynamic_limit}. 
If it holds that
$$\la x, \mu^N_0\ra =\frac{W_N}{N}\to 0, \quad \mbox{ as } N\to\infty,$$
(which is in particular true if $w_0<\infty$), then, we have that $\displaystyle\lim_{N\to \infty} \mu^N_t \stackrel{\P}{=} \delta_0$ in probability for all times $t$.
\end{proposition}

\begin{proposition}[Mean field limit of the empirical wealth under equilibrium measures.]
\label{lem:equilibria}
Suppose $\mu^{\infty, N}_0 \sim \mathrm{Unif}[\Delta_{N-1}]$ (therefore we assume the total wealth is fixed and equal to 1) for each $N \in \N$ and consider the empirical measure on $\R_+$
\[
\mu^N_0 = \frac{1}{N} \sum_{i=1}^N \delta_{X^{i, N}_0},  \quad (X^{1, N}_0 ,\ldots, X^{N, N}_0) \sim \mu^{\infty, N}_0.
\]
Then as $N\to \infty$, 
\[
\mu^N_0 \Longrightarrow \delta_0, \quad a.s.
\] 
In particular the assumptions of Theorem \ref{th:hydrodynamic_limit} hold and, since $w_0 = 1$, Proposition \ref{cor:delta} is in effect.
\end{proposition}

\begin{corollary} 
\label{cor:kinetic_and_equilibria}
Let $\displaystyle \lim_{N \to \infty} W_N = w_0 \in (0, \infty]$. Assume that $\mu_t$ is a solution of \eqref{eq:kinetic_equation} which has has a density $f_t$ for all $t$.
Then 
\begin{enumerate}
\item If $w_0 = \infty$, the exponential distributions 
\be \label{eq:equilibria}
\tilde f(x) = \frac{e^{-x/m}}{m},
\ee
are equilibria for the operator $Q$ and remain invariant under \eqref{eq:kinetic_equation}. In particular, if $f_0$ is of the form \eqref{eq:equilibria} with
$\la x, f_0 \ra = m_0>0,$ then the distribution \eqref{eq:equilibria} with $m= m_0$ is a stationary solution of \eqref{eq:kinetic_equation}. 

\item If $0<w_0<\infty$, then the following distributions are compactly supported on $[0,w_0]$ and are equilibria for the operator $Q$
\be \label{eq:equilibria_bounded}
\tilde f(x)  = \frac{e^{-x/m}}{m(1-e^{-w_0/m})} \mathbbm1_{\{ x\leq w_0\}}.
\ee
\item (Uniqueness of the invariant family at $w_ 0 = \infty$) Moreover, under the extra assumption that the density $f_t$ is differentiable on $\R_+$, then measures with density \eqref{eq:equilibria} are the unique equilibria. 
\end{enumerate}
\end{corollary}

The next proposition tells us that invariant distributions \eqref{eq:equilibria_bounded} cannot be obtained as limits of the discrete measures, therefore they are extraneous, while invariant distributions of the form \eqref{eq:equilibria} are possible.

\begin{proposition} \label{prop:2am} Let $\displaystyle \lim_{N \to \infty} W_N = w_0 \in (0, \infty]$.
\begin{enumerate} 
	\item ($w_0 = \infty$) Consider an infinite i.i.d. sequence $\{ X_i \}_{i \ge 1}$ of $\mathrm{Exp}(1/m_0)$ variables. For every $N \in \N$, define
		\[
		W_N = \sum_{i=1}^N X_i, \quad \text{and} \quad \mu^N_0 (x) = \frac{1}{N} \sum_{i = 1}^N \delta_{X_i}(x), 
		\]
		i.e. the initial wealth of each agent is an independent exponential random variable as we increase the number of agents, but always fixed across the $N$.
		Then as $N \to \infty$, $\mu^N_0 \Longrightarrow \mu_0$ where $\mu_0(x) = \frac{1}{m_0}e^{-x/m_0} \, dx$ and therefore Theorem \ref{th:hydrodynamic_limit} holds. Then by Corollary \ref{cor:kinetic_and_equilibria}, $\mu_0$ remains invariant in time.  
	
	\item ($w_0 < \infty$)  There does not exist a sequence of measures $\{ \mu_0^N \}_{N \in \N}$ so that  $\mu^N_0 \Longrightarrow \mu_0$ with $\mu_0$ having a density \eqref{eq:equilibria_bounded}. 	
\end{enumerate}
\end{proposition}

Finally, in Section \ref{sec:partitions} we have results when our process is viewed as a process on integer partitions of numbers. We omit from listing them here for readability reasons, but the interested reader can directly look in the section for Theorems \ref{thm:partG}, \ref{thm:partD} and \ref{thm:partE}.

\section{Process level convergence to a discrete-time continuous space model}
\label{sec:procon}
This section is dedicated to proving the process level convergence of the DS-DT model to the CS-DT model, 
thus completing Proposition 7.3 in \cite{DGS} where convergence of the one dimensional marginals was shown. 
We need an equivalent, alternative description of the DS-DT model, so we begin this section with it. 
The number of agents $N$ remains fixed throughout this section, so we will omit it from the notation, and we will write 
${\bf Y}^{(n)}$, ${\bf X}^{(n)}$  and ${\bf X}^{(\infty)}$ instead of ${\bf Y}^{n, N}$,  ${\bf X}^{n, N}$ and ${\bf X}^{\infty, N}$ respectively.

\begin{proof}[Proof of Theorem \ref{thm:processlevel}]

Since we may embed the sequence of processes $\{{\bf Y}^{(n)}\}_{n \in \N}$ in $(\Delta_{N-1})^\N$, which is compact, the collection of their induced measures $\{ \mathcal P^n\}_{n\in \N}$ is tight. Therefore, for process-level convergence,  it suffices to show that finite dimensional marginals converge weakly. We show this for vectors of the form $(\Y^{(n)}_0, \Y^{(n)}_1, \ldots, \Y^{(n)}_k)$ with law denoted by \eqref{eq:induced}.  In the calculation below, we denote by $\nu_U$ the law of the random variable $U$. For any generic measure $\mu$, we denote by $\E_{\mu}$ the expectation operator with respect to that measure. 

For any bounded continuous function $f$
 
 \allowdisplaybreaks
\begin{align*}
	\E_{\mathcal P^n_k}(&f(\Y^{(n)}_0, \Y^{(n)}_1, \ldots, \Y^{(n)}_k)) = \sum_{\y_0}  \cdots \sum_{\y_k} f(\y_0, \ldots, \y_k) \mathcal P^n_k\{ \Y^{(n)}_0 = \y_0, \ldots, \Y^{(n)}_k = \y_k \}\\
	&= \sum_{\y_0}  \cdots \sum_{\y_k} f(\y_0, \ldots, \y_k)  \mathcal P^n_{k-1}\{ \Y^{(n)}_0 = \y_0, \ldots, \Y^{(n)}_{k-1} = \y_{k-1}\} \P\{ \Y^{(n)}_k = \y_k | \Y^{(n)}_{k-1} = \y_{k-1} \} \\
	&= \sum_{\y_0}  \cdots \sum_{\y_{k-1}} \mathcal P^n_{k-1}\{ \Y^{(n)}_0 = \y_0, \ldots, \Y^{(n)}_{k-1} = \y_{k-1}\} \\
	&\phantom{xxxxxxxxxxxxx}\times \sum_{\y_k} f(\y_0, \ldots, \y_k)\P\{ \Y^{(n)}_k = \y_k | \Y^{(n)}_{k-1} = \y_{k-1} \} \\
	&= \sum_{\y_0}  \cdots \sum_{\y_{k-1}}   \mathcal P^n_{k-1}\{ \Y^{(n)}_0 = \y_0, \ldots, \Y^{(n)}_{k-1} = \y_{k-1}\} \\
	&\phantom{xxxxxxxxxxxxx}\times \frac{1}{N(N-1)} \sum_{(i,j): i \neq j} \int_{0}^1f(\y_0, \ldots, \y_{k-1},  g_{i,j}(\y_{k-1}, u))\,du\\	
		&=\frac{1}{N(N-1)}  \sum_{\y_0}  \cdots \sum_{\y_{k-1}}  \mathcal P^n_{k-1}\{ \Y^{(n)}_0 = \y_0, \ldots, \Y^{(n)}_{k-1} = \y_{k-1}\} \\
		&\phantom{xxxxxxxxxxx}\times \sum_{(i,j): i \neq j}  \E_{\nu_{U^{(n)}_{i,j}(k-1)}} \Big(f(\y_0, \ldots, \y_{k-1},  g_{i,j}(\y_{k-1}, U^{(n)}_{i, j}(k-1))\Big) \\
		&= \frac{1}{N(N-1)} \E_{ \mathcal P^n_{k-1}} \bigg( \sum_{(i,j): i \neq j}  \E_{\nu_{U^{(n)}_{i,j}(k-1)}} \Big( f(\Y^{(n)}_0, \ldots, \Y^{(n)}_{k-1},  g_{i,j}(\Y^{(n)}_{k-1}, U^{(n)}_{i, j}(k-1)))\Big)\\
		&=\frac{1}{N(N-1)}  \sum_{(i,j): i \neq j} \E_{\mathcal P_{k-1}^n \otimes \nu_{U^{(n)}_{i,j}(k-1)}} \Big(f(\Y_0, \ldots, \Y_{k-1},  g_{i,j}(\Y_{k-1}, U^{(n)}_{i,j}(k-1)))\Big).
\end{align*}
At this point, we have to deal with a small technical issue. The function $g_{i,j}$ is not immediately continuous on its arguments, since it can create jumps of order $1/n \ge U_{i,j}^{(n)}(k)(y_i(k) + y_j(k)) -[U_{i,j}^{(n)}(k)(y_i(k) + y_j(k)) ]_n$ for all $k$.
However $f$ is a bounded continuous function on a compact space $\Delta_{N-1}^{k+1}$, and it is uniformly continuous in the last coordinate. Then define on $\Delta_{N-1}\times[0,1]$ the  bounded continuous function
\[
g_{i,j}^{\text{cont}}(\y, u ) =   (y_1, \ldots, u(y_i + y_j), \ldots,  (1-u)(y_i + y_j), \ldots, y_N).
\]
Fix a $\delta> 0$ and let $n = n(\delta)$ be large enough so that 
\[
\sup_{\y \in \Delta_{N-1}(n)}\sup_{u \in [0,1]} \sup_{(i,j)}\| g_{i,j}^{\text{cont}}( \y,u) -  g_{i,j}(\y, u )\|_{\infty} < \delta.
\]
Fix an $\e>0$ and choose $\delta$ so that for any $\|(\z_1, \ldots, \z_k) - (\y_1, \ldots, \y_k)\|_{\infty} < \delta$
\[ 
\| f(\z_1, \ldots, \z_k) - f(\y_1, \ldots, \y_k)\|_{\infty} < \e (kN)^{-2}. 
\]  
Then we proceed with the computation for $n$ large enough:
\begin{align*}
\E(f(&\Y^{(n)}_0, \Y^{(n)}_1, \ldots, \Y^{(n)}_k)) \\
&= \frac{1}{N(N-1)} \sum_{(i,j): i \neq j} \E_{\mathcal P_{k-1}^n\otimes \nu_{U^{(n)}_{i,j}(k-1)}} \Big(f(\Y_0, \ldots, \Y_{k-1},  g_{i,j}(\Y_{k-1}, U^{(n)}_{i,j}(k-1)))\Big)\\
&= \frac{1}{N(N-1)} \sum_{(i,j): i \neq j} \E_{\mathcal P_{k-1}^n \otimes \nu_{U^{(n)}_{i,j}(k-1)}} \Big(f(\Y_0, \ldots, \Y_{k-1},  g^{\text{cont}}_{i,j}(\Y_{k-1}, U^{(n)}_{i,j}(k-1)))\Big)+ O(\e) \\
&= \frac{1}{N(N-1)} \sum_{(i,j): i \neq j} \E_{\mathcal P_{k-1}^n \otimes \nu_{U^{(n)}_{i,j}(k-1)}} \Big(\tilde f_{i,j}(\Y_0, \ldots, \Y_{k-1}, U^{(n)}_{i,j}(k-1))\Big)+ O(\e).
\end{align*} 
Above, $\tilde f_{i,j}$ is a bounded continuous function. By iterating the same argument using the Markov property iteratively, we conclude, for $n$ large enough that 
\begin{align*}
\E(f(&\Y^{(n)}_0, \Y^{(n)}_1, \ldots, \Y^{(n)}_k)) = \left( \frac{1}{N(N-1)} \right)^{k} \\
&\phantom{x} \times \!\!\!\!\sum_{\stackrel{(i_0,j_0)}{ i_0 \neq j_0}}\!\! \cdots \sum_{\stackrel{(i_{k-1},j_{k-1})}{ i_{k-1}\neq j_{k-1}}} \E_{\mu_0^{(n)}\bigotimes_{\ell=0}^{k-1}\nu_{U^{(n)}_{i_{\ell},j_{\ell}}(\ell)} }\Big(\tilde f_{(i_0,j_0), \ldots,(i_{k-1}, j_{k-1})}(\Y_0, U^{(n)}_{i_{0},j_{0}}(0), \ldots, U^{(n)}_{i_{k-1},j_{k-1}}(k-1))\Big)\\
&\phantom{xXX}+ O(\e).
\end{align*} 
The sums above are finitely many, so the accumulated error is bounded by $C\e$. Each multiindexed $\tilde f$ is a bounded continuous function on all its arguments. Finally, the assumptions of the theorem imply the joined weak convergence 
\[
(\Y_0, U^{(n)}_{i_{0},j_{0}}(0), \ldots, U^{(n)}_{i_{k-1},j_{k-1}}(k-1)) \Longrightarrow({\bf X}^{(\infty)}_0, U^{(\infty)}_{i_{0},j_{0}}(0), \ldots, U^{(\infty)}_{i_{k-1},j_{k-1}}(k-1)).
\]
The limiting vector can be used to uniquely construct the CS-DT process using the indices of the associated function $\tilde f$. By reversing the decomposition above,   therefore 
\be
\big|\lim_{n\to \infty} \E_{\mathcal P^n}(f(\Y^{(n)}_0, \Y^{(n)}_1, \ldots, \Y^{(n)}_k)) -  \E_{\mathcal P^\infty}(f({\bf X}^{(\infty)}_0, {\bf X}^{(\infty)}_1, \ldots, {\bf X}^{(\infty)}_k))\big| = O(\e). 
\ee
Let $\e \to 0$ to finish the proof. 
\end{proof} 

\begin{remark}[Almost sure convergence for finite sample paths] 
Assume that the initial distributions satisfy  $\Y^{(n)}_0 \to {\bf X}^{(\infty)}_0$ a.e. as $n\to \infty$ and that we use common uniforms for each time step $k$, i.e.
\[
 U_{i,j}(k) \equiv U_{i,j}^{(n)}(k) = U_{i,j}^{(m)}(k)=U_{i,j}^{(\infty)}(k), \quad \text{ for all } n,m \in \N,
\]
while maintaining the independence across the time index.
Then for any fixed $k \in \N$
\[
(\Y^{(n)}_0, \ldots, \Y^{(n)}_k) \stackrel{a.s.}{\longrightarrow} ({\bf X}^{(\infty)}_0, \ldots, {\bf X}^{(\infty)}_k),
\]
provided the same indices $(i,j)$ are selected at each step. 
This is because of the compact state space for these processes. For any fixed $n$, the construction using now the common (in $n$) uniform random variables $U_{i,j}^{(n)}(\ell)$ creates an error of at most $2/n$ per step in the supremum norm of the state space, so the total error is $2k/n$, which vanishes as $n \to \infty$. 
\end{remark}

\section{Kinetic equations as thermodynamic limit of the Markov chain with continuous state space}
\label{sec:sara2}

We devote this section to proving that  equation \eqref{eq:kinetic_equation} is obtained as the limit in probability of \eqref{eq:martingale_general} as $N\rightarrow \infty$, (see Theorem \ref{th:hydrodynamic_limit}). Before stating the result rigorously, we need to mention some terminology and basic facts.

\begin{definition}[Solutions] We say that a measure $(\mu_t)_{t<T}$ is local solution if it satisfies \eqref{eq:kinetic_equation} for all functions $f$ which are bounded and measurable. If $T$ can be taken to be $+\infty$, then we say we have a (global) solution of \eqref{eq:kinetic_equation}.
\end{definition}

It is important to ascertain that solutions do exist, and this is the content of the next proposition. The proof of it follows the same arguments as in the proof for Smoluchowski's equation in \cite[Proposition 2.2]{norris1999smoluchowski}, and it is omitted from this manuscript.      

\begin{proposition}[Existence and uniqueness of solutions]
Suppose that $\mu_0\in \mathcal{M}_1(\R_+)$. The kinetic equation \eqref{eq:kinetic_equation} has a unique solution $(\mu_t)_{t\geq 0}$ with initial data $\mu_0$. 
\end{proposition}

Above we introduced $\mathcal{M}_1(\R_+)$ as the space of probability measures with support on the non-negative reals.
In general, $\mathcal{M}_1(K)$ denotes the set of probability measures on the set $K$. We have already discussed how the empirical measure $\mu_t^N \in  \mathcal{M}_1(\R_+)$. In particular, for any $t \ge 0$,  $\mu_t^N$ is a \emph{ random  element} of  $\mathcal{M}_1(\R_+)$, and its distribution is solely dictated by the distribution of the Markov chain at time $t$. 

The next proposition states the two main conservation properties that we are using throughout the manuscript. First we show that the support of the initial measure dictates the support of all $\mu_t$ without exiting the class of probability measures, and the second property is the conservation of total wealth. Recall the notation introduced in Section \ref{sub:mgf}.  

\begin{proposition} Suppose that $w_0<\infty$.
Assume that $\mu_0^N \in \mathcal{M}_1([0, w_0])$, then $\mu_t^N \in \mathcal{M}_1([0, w_0])$ for all times. Moreover, if $\la x,\mu_0^N\ra = m_0 \in \R_+$, then $\la x, \mu_t^N\ra = m_0$ for all times.
\end{proposition}
\begin{proof}
To check the proposition one just needs to notice that
\[ \la \mathbbm1\{x\leq w_0\}, Q(\mu)\ra =0,\]
for any measure $\mu$. Therefore, by \eqref{eq:kinetic_equation}, we have that
\[ \la \mathbbm1\{x\leq w_0\}, \mu_t^N \ra = \la \mathbbm1\{x\leq w_0\}, \mu_0^N\ra=1,\]
and so $\mu_t^N(\{ x\ ; \ x\leq w_0\}) =1$. 
This implies we can write $\mu_t^N \in \mathcal{M}_1([0,w_0])$ for any $t$. 

The second statement can be proven analogously substituting $g(x)= x\mathbbm1\{x\leq w_0\}$ in \eqref{eq:def_QN}. 
\end{proof}

The symbol  $D(K, S)$ denotes the space of c\`adl\`ag (right continuous with left limit) functions from $K$ to $S$, called the Skorokhod space.  
We wish we to study the process of the empirical measures $\{ \mu_t^N \}_{t \ge 0}$ as a sequence in $N$. 
For any fixed $N$, the sequence $\{ \mu_t^N \}_{t\ge 0}$ is an element of $D([0,\infty); \mathcal{M}_1(\R_+))$. 
All necessary background information for Skorokhod spaces that will be used in the section can be found in the Appendix.

With the notation set, we can now proceed and prove theorem \ref{th:hydrodynamic_limit}. Technical proofs are left to the end of the section to not mar the exposition. Again, recall the notation from Section \ref{sub:mgf}.

\subsection{Proof of Theorem \ref{th:hydrodynamic_limit}}
\bigskip
%
 The main idea for the proof 
is to take the limit as $N\to \infty$ in the martingale formulation \eqref{eq:martingale_general} by following the methodology presented in \cite{norris1999smoluchowski}. 

The theorem can be proven directly from the following three propositions. We do that right after these propositions are proven.

\begin{proposition}[Martingale convergence]
\label{prop:convergence_martingale}
For any $g\in C_b(\R_+)$, $t\geq 0$, it holds that
\[
\lim_{N \to \infty} \sup_{0 \leq s\leq t} |M^{g, N}_s| = 0\qquad \mbox{in } \mathcal L^2(\R),
\]
where $M^{g,N}_t$ is defined in \eqref{eq:martingale_general}.
In particular, the limit also holds in probability.
\end{proposition}

\begin{proposition}[Weak convergence for the measures]
\label{prop:almost_sure_convergence_for_measures}
The sequence of laws $\mathcal P_N$ of the elements $\{\mu_t^N\}_{t \in \R_+}$ is tight. Therefore there exists a weakly convergent subsequence  $(\mu^{N_k})_{k\in\NN}$ in $D([0,\infty); \mathcal{M}_1(\R_+))$ as $k~\rightarrow~\infty$.
\end{proposition}

\begin{proposition}[Convergence for the trilinear term]
\label{prop:convergence_trilinear_term}
For any converging subsequence $\{ \mu^{N_k} \}_{k \in \N}$
(and particularly for those established in Proposition \ref{prop:almost_sure_convergence_for_measures}), 
it holds that
$$\int^t_0 \langle f, Q^{(N_k)}(\mu^{N_k}_s)\rangle \, ds \rightarrow \int^t_0 \langle f, Q(\mu_s)\rangle \, ds \quad \mbox{weakly,}$$
as $k\to \infty$.
\end{proposition}
			
\subsubsection{Proof of Proposition \ref{prop:convergence_martingale}.}

Keep in mind that $M^{g, N}_t$ is a martingale. From Proposition 8.7 in \cite{darling2008differential} (a consequence of Doob's $\mathcal L^2$ inequality) we have that for any finite $T$,
\be \label{eq:martingale_estimate_norris}
\E  \left[ \sup_{s\leq T} |M^{g, N}_s|^2\right] \leq 4 \E \int^T_0 \alpha^{g, N}(\mu^{(2, N)}_s) ds,
\ee
where in this case
\begin{align} 
\alpha^{g, N}(\mu^{(2, N)}_s) &= \int_{[0,1]}\int_{\R_+^2}\lp \frac{1}{N} \Big( g(r(x+y))+g((1-r)(x+y))-g(x)-g(y)\Big) \rp^2\nonumber \\
& \phantom{xxxxxxxxxxxxxxxxxxxxxxxxxxxx}\times\ \mathbbm1_{\{x+y\leq W_N\}}\ N\mu^{(2,N)}_s(dx,dy)\, dr \notag\\
& \le \frac{N}{N^2} \frac{N-1}{N} 16 \| g\|_{\infty}^2 \le \frac{16}{N} \| g\|_{\infty}^2. 
\label{eq:alpha_previsible} 
\end{align}
Use this estimate in  \eqref{eq:martingale_estimate_norris}  to obtain
\begin{equation} \label{eq:bound_martingale}
\E  \left[ \sup_{s\leq T} |M^{g, N}_s|^2\right]  \leq \frac{1}{N}64 \|g\|^2_{\infty}T.
\end{equation}
This gives the convergence of the supremum towards 0 in $\mathcal L^2$ as $N \to \infty$, which implies also the convergence in probability. \qed

\subsubsection{Proof of Proposition \ref{prop:almost_sure_convergence_for_measures}}
\label{sec:step2_convergence_of_measures}

The results stated in the proposition will be proven at the end of this subsection, and they follow from two lemmas.
\begin{lemma} 
\label{lem:tightness_for_the_action}
Fix an $f \in  C_b(\R_+)$. Then the sequence of laws  of $(\langle f, \mu^N \rangle)_{N\in\NN}$ on $D([0,\infty); \R_+)$ is tight.			
\end{lemma}

\begin{proof}[Proof of Lemma \ref{lem:tightness_for_the_action}]
 We use  Theorem \ref{th:criteria_tightness_Skorokhod} in the Appendix. Thus, we need to verify the two conditions of the Theorem.   
 
 To prove condition $(i)$ of the Theorem we use that for any fixed $f \in C_b(\R_+)$
\[
|\la f, \mu^N_t \ra | = \left| \frac{1}{N}\sum^N_{i=1}f(X^{i,N}_t)\right| \leq  \frac{1}{N} \sum_{i=1}^N|f(X^{i, N}_t)|\leq \|f\|_\infty 
\]
so for all $t\geq 0$, $\la f, \mu^N_t \ra \in  [-\|f\|_\infty, \|f\|_\infty]$. To directly see the connection with Theorem  
\ref{th:criteria_tightness_Skorokhod}, set $\Lambda_{\eta, t} = [-\|f\|_\infty, \|f\|_\infty]$ (fixed for any $\eta$) and $X_N(t) = \la f, \mu^N_t \ra$.

The verify the second condition $(ii)$ of Theorem \ref{th:criteria_tightness_Skorokhod} we make use of the following inequalities:
\begin{equation} \label{eq:4bound_square_martingale}
\E \left[ \sup_{r\in[s,t)}|M_{r}^{f, N}-M_s^{f, N}|^2\right] \leq \frac{1}{N} 64 \|f\|^2_\infty (t-s)
\end{equation}
 and 
\begin{equation} \label{eq:4bound_square_operator}
 \E \left[ \sup_{r\in[s,t)} \left( \int^r_s \langle f, Q^{(N)}(\mu^N_u) \rangle \, du \rp^2 \right]\leq 16 \|f\|^2_\infty (t-s)^2.
 \end{equation}
 To see inequality \eqref{eq:4bound_square_martingale} recall that since $M^{f, N}_t$ is an $\mathcal F_t$-martingale, then $\tilde M^{f, N}_t = M^{f, N}_{t+s} -  M^{f, N}_{s} $ is an   $\tilde{\mathcal F}_t =  \mathcal F_{t+s}$
martingale. Therefore
 \begin{align*}
 \E \left[ \sup_{r\in[s,t)}|M_{r}^{f, N}-M_s^{f, N}|^2\right] =  \E \left[ \sup_{r\in[0,t-s)}|\tilde M_{r}^{f, N}|^2\right] \le  \frac{1}{N}64 \|f\|^2_{\infty}(t-s),
 \end{align*}
 just like in equation \eqref{eq:bound_martingale}.
 Inequality \eqref{eq:4bound_square_operator}  follows from \eqref{eq:def_QN} and a bound similar to the one used in \eqref{eq:alpha_previsible}.
 
Equations \eqref{eq:4bound_square_martingale} and \eqref{eq:4bound_square_operator}
together give  the bound
\begin{equation} \label{eq:4bound_square_measure}
\E \left[ \sup_{r\in[s,t)}|\langle f, \mu^N_r-\mu^N_s\rangle|^2\right] \leq A\lp(t-s)^2+\frac{(t-s)}{N} \rp
\end{equation}
for some $A>0$ depending only on $\|f\|_\infty$. 
With these estimates the proof follows as in \cite{merino2016isotropic} where further details can be found.
\end{proof}

\begin{lemma} 
\label{lem:tightness_for_the_measures}
The sequence of laws $\{ \mathcal P_N\}_{N \in \NN}$ of the elements  $(\mu^N_t)_{t \in \R_+}  \in D([0,\infty); \mathcal{M}_1(\R_+))$ is tight.
\end{lemma}

\begin{proof}[Proof of Lemma \ref{lem:tightness_for_the_measures}]
We will use Theorem \ref{th:jakubowski_criteria} in the Appendix  to prove this result. To check condition $(i)$, we find a suitable compact set $W\in \mathcal{M}_{\le1}(\R_+)$, where $\mathcal{M}_{\le1}(\R_+)$ is the set of all sub-probability measures on $\R_+$, which is a separable, compact metric space, and therefore a completely regular topological space. Any closed subset $W$ of $\mathcal{M}_{\le1}(\R_+)$ will be compact with respect to the topology induced by the weak convergence of measures, and it will be metrizable as a subset of a metric space. 

We define for some positive constant $C$ the set 
\[ W_C:=\left\{ \tau\in \mathcal{M}_1(\R_+)\,:\, \int_{\R_+} x\,  \tau(dx) \leq C \right\}, \]
which is closed (and therefore compact). Assume that $\{ \tau_n \}_{n \in \N}$ is a sequence of measures in $W_C$ that converge weakly to $\tau$. Then for any $M \in \R_+$
\[
\int_{\R_+} x \tau(dx) = \lim_{M \to \infty} \int_{\R_+} (x \wedge M) \tau(dx) =  \lim_{M \to \infty}  \lim_{n \to \infty} \int_{\R_+} (x \wedge M) \tau_n(dx) \le 
\lim_{n \to \infty} \int_{\R_+} x \tau_n(dx) \le C,
\] 
and therefore the limit point $\tau$ is also in $W_C$.

In our case, from the conservation of the total mass (wealth) and the fact that we can find a $c_1$ so that $W_N < c_1N$, for all $N \in \N$, we have 
$$\int_{\R_+} x \mu_t^N(dx) = \frac{1}{N}\sum_{i=0}^N X_t^{i,N} =\frac{1}{N}\sum_{i=0}^N X_0^{i, N} = \int_{\R_+} x \mu_0^N(dx) \leq c_1 \quad \mbox{a.s.}$$
Consider $\lp\mathcal{P}_N\rp_{N\in\NN}$ the family of probability measures in $\mathcal{M}_1(D([0,\infty); W_{c_1}))$ which are the laws of $(\mu^N_t)_{t\in \R_+}$. We have that
\[ \mathcal{P}_N(D([0,\infty); W_{c_1}) =1 \quad \mbox{ for all } N\in \NN. \]
This verifies condition $(i)$ of Theorem \ref{th:jakubowski_criteria}.

In order to check condition $(ii)$ we will use the family of continuous functions on $\mathcal{M}_{ \le1}(\R_+)$ defined as
$$\mathbb{F}= \{ F\, :\, \mathcal{M}_{\le1}(\R_+) \rightarrow \R \, :\, F(\tau)= \langle f, \tau \rangle \mbox{ for some } f\in C_b(\R_+)\}.$$
This family is closed under addition since $C_b(\R_+)$ is, it is continuous in $\mathcal{M}_{\le 1}(\R_+)$, and separates points in $\mathcal{M}_{\le 1}(\R_+)$: if $F(\tau)=F(\bar \tau)$ for all $F\in \mathbb{F}$ then
\[ 
\int_{\R_+} f(x) d(\tau-\bar \tau) (x) =0 \quad \forall f\in C_b(\R_+) 
\]
 hence $\tau\equiv\bar\tau$, since we can approximate indicator functions for any Borel set $A$ using functions from $C_{b}(\R_+)$. 
So we are left with proving that for every $f\in C_b(\R_+)$ the sequence $\{\langle f, \mu^N \rangle\}_{N\in \NN}$ is tight. This was proven in Lemma \ref{lem:tightness_for_the_action}. 
\end{proof}

Now Proposition \ref{prop:almost_sure_convergence_for_measures} follows immediately.

\begin{proof}[Proof of Proposition \ref{prop:almost_sure_convergence_for_measures}] The result follows from
Lemma \ref{lem:tightness_for_the_measures} and  Prokhorov's theorem. 
\end{proof}

\bigskip
%

To prove Proposition \ref{prop:convergence_trilinear_term} we need the following three lemmas. Throughout we are assuming that $\{ \mu^{N_k} \}$ is a converging sequence in the space $D([0, \infty); \mathcal M_1(\R_+))$.

\begin{lemma}[Continuity of the limit]
\label{lem:continuity_of_limit}
The weak limit of $(\mu^{N_k}_t)_{t\geq 0}$ as $k\rightarrow\infty$  is continuous in time a.e..
\end{lemma}

\begin{proof}[Proof of Lemma \ref{lem:continuity_of_limit}]
We have that for any $f\in C_b(\R_+)$
\[ |\langle f, \mu^{N_k}_t \rangle- \langle f, \mu^{N_k}_{t-}\rangle | \leq\frac{4}{N_k}  \|f\|_{\infty}, \]
when a jump happens in the process only the wealth of two individuals is altered.
 Then we may apply Theorem \ref{th:continuity_criteria_limit_Skorokhod_space} of the Appendix to obtain that $\la f, \mu_t\ra$ is continuous for any $f\in C_b(\R_+)$ and this implies the continuity of $(\mu_t)_{t\geq0}$. 
\end{proof}

\begin{lemma}[Uniform convergence]
\label{lem:uniform_convergence_in_time}
For all $f\in C_b(\R_+)$,  and finite $t\ge 0$ we have
$$\sup_{s\leq t}|\la f, \mu^{N_k}_s-\mu_s\ra|\to 0 \quad
\mbox{weakly}$$
as $k\to\infty$. 
\end{lemma}

\begin{proof}[Proof of Lemma \ref{lem:uniform_convergence_in_time}]
By Lemma \ref{lem:continuity_of_limit}, the limit of $(\mu_t^{N_k})_{N\in \NN}$ is continuous in time. The statement is consequence of the Continuous Mapping Theorem in the Skorokhod space  and the fact that $g(X)(t)=\sup_{s\leq t} |X|$ is a continuous function in this space. 
\end{proof}

\begin{lemma}
\label{lem:limit_counting_measures}
For all $f\in C_b(\R_+)$,  and finite $t\ge 0$ we have
$$\sup_{s\leq t} |\langle f, Q^{(N_k)}(\mu^{N_k}_s)- Q(\mu_s)\rangle| \rightarrow 0 \quad \mbox{weakly}$$
as $k\rightarrow \infty$.
\end{lemma}

\begin{proof}[Proof of Lemma \ref{lem:limit_counting_measures}]
We abuse notation and denote by $(\mu^N_t)_{N\in \NN}$ the convergent subsequence.
The result will manifest itself when we show that for  all $f\in C_b(\R_+)$:
\begin{itemize}
	\item[(i)] $\sup_{s\leq t}|\langle f, \lp Q-Q^{(N)}\rp (\mu^{N}_s)\rangle| \rightarrow 0$ as $N\rightarrow \infty$,
	\item[(ii)] $\sup_{s\leq t}\left|\langle f,  Q\lp \mu^N_s \rp -Q\lp\mu_s\rp \rangle\right| \rightarrow 0$ as $N\rightarrow \infty$. 
\end{itemize}
We will use the fact that the product measures also converge weakly, i.e.
$\mu^N_t \otimes \mu^N_t \Longrightarrow \mu_t \otimes \mu_t$.

Item (i) is then a consequence of 
\begin{align} \nonumber
&|\langle f, \lp Q-Q^{(N)}\rp (\mu^{N}_s)\rangle| \\
&\le \left|\int_{[0,1]}\int_{\R_+}\int_{\R_+}[f(r(x+y)) + f((1-r)(x+y)) -f(x) - f(y) ] d (\mu^N_s\otimes\mu^N_s - \mu^{(2,N)}_s ) dr\right|\notag\\
&\leq 4 \| f\|_{\infty} \left| \int_{\R_+^2} d (\mu^N_s\otimes\mu^N_s - \mu^{(2,N)}_s )  \right| =  4 \| f\|_{\infty}  \left| \frac{1}{N} \int_{\R_+} d\mu^N_s   \right| \leq\frac{4}{N}\|f\|_{\infty} . \label{eq:4bound_difference_trilinear_term}
\end{align}
The bound is true for any $s$ and therefore for the supremum up to a finite time as well.
Now for (ii), we compute 
\begin{align} \nonumber
\sup_{s\leq t}&\left|\langle f,  Q( \mu^N_s) -Q(\mu_s) \rangle\right| \\
 &\leq \sup_{s\leq t}  \int_{[0,1]}\int_{\R^2_+}  \left|f(r(x+y)) +f((1-r)(x+y)) - f(x) - f(y) \right| \notag \\
& \phantom{xxxxxxxxxxxxx}\left| \mathbbm1_{\{x+y \leq W_N\}}\mu^N_s(dx)\mu^N_s(dy)-\mathbbm1_{\{x+y \leq w_0\}}\mu_s(dx)\mu_s(dy)\right|\, dr \notag\\
 &\le 4\| f \|_{\infty}  \sup_{s\leq t} \int_{\R^2_+}\left| \mathbbm1_{\{x+y \leq W_N\}}\mu^N_s(dx)\mu^N_s(dy)-\mathbbm1_{\{x+y \leq w_0\}}\mu_s(dx)\mu_s(dy)\right| \notag\\
 &\le 4\| f \|_{\infty}  \sup_{s\leq t} \int_{\R^2_+}\left| \mathbbm1_{\{x+y \leq w_0\}}\mu^N_s(dx)\mu^N_s(dy)-\mathbbm1_{\{x+y \leq w_0\}}\mu_s(dx)\mu_s(dy)\right| \notag\\
  &\le 4\| f \|_{\infty}  \sup_{s\leq t} \int_{\R^2_+}\left|\mu^N_s(dx)\mu^N_s(dy)- \mu_s(dx)\mu_s(dy)\right|.
\label{eq:4bound_trilinear_term_uniformly}
\end{align} 
We conclude $(ii)$  with an argument analogous to Lemma \ref{lem:uniform_convergence_in_time} applied to the function $f =1$. 
\end{proof}

\begin{proof}[Proof of Proposition \ref{prop:convergence_trilinear_term}]
By Lemma \ref{lem:limit_counting_measures} we can pass the limit inside the time integral.  
\end{proof}

We are now in position to prove Theorem \ref{th:hydrodynamic_limit}:
	
\begin{proof}[Proof of Theorem \ref{th:hydrodynamic_limit}]
	
The weak form of item $(A)$ is proven in Lemma  \ref{lem:uniform_convergence_in_time}, item $(B)$ is proven  in Proposition \ref{prop:convergence_martingale}, and item $(C)$ is the content of Proposition \ref{prop:convergence_trilinear_term}. Since all those weak convergences in the previous propositions were to 0, they can be upgraded to convergence in probability.

Then (and also by using the assumptions of the theorem) we have that for any $f \in C_b(\R_+)$ and any converging subsequence of measures,
\begin{align*}
0 &\stackrel{\mathcal D}{=} \lim_{N\to \infty}  M_t^{g, N}  \\
&=   \lim_{N\to \infty} \langle  g, \mu^N_t \rangle -  \langle g, \mu^N_0\rangle - \int_0^t \langle g, Q^{(N)}(\mu^N_s) \rangle ds \stackrel{\mathcal D}{=} \langle  g, \mu_t \rangle -  \langle  g, \mu_0\rangle - \int_0^t \langle  g, Q(\mu_s) \rangle,  
\end{align*}
and therefore the limit of the subsequence of measures  must satisfy equation  \eqref{eq:kinetic_equation}. 
Using the uniqueness of the kinetic equation \eqref{eq:kinetic_equation}, we have that all the convergent subsequences from Proposition \ref{prop:almost_sure_convergence_for_measures} converge to the same limit. Hence the whole sequence converges (if a tight sequence has every weakly convergent subsequence converging to the same limit, then the whole sequence converges weakly to that limit \cite{Billingsley}). 
	
Now, we have that the weak limit of $(\mu_t^N)_{N\in\NN}$ satisfies the kinetic equation \eqref{eq:kinetic_equation} (thanks to Prop. \ref{prop:convergence_martingale}, \ref{prop:convergence_trilinear_term}), so it is deterministic. Therefore, we actually have convergence in probability.
\end{proof}

\section{Invariant measures for the mean field limit}
\label{sec:sara}

In this section we discuss the invariant measures. 

\begin{proof}[Proof of Proposition \ref{cor:delta}]
If $\la x, \mu_0^N\ra \to 0$ as $N\to\infty$, by positivity of the support of the measures  and conditions \eqref{eq:bound_initial_data}- \eqref{eq:limit_initial_data}, it follows that
$$\la x, \mu_0\ra =0.$$
On the other hand,  $\mu_0$ is a probability measure, so the above implies that $\mu_0(x) = \delta_0(x)$. Then it follows that $\mu_t(x)=\delta_0(x)$ since we already argued that the delta distribution is an invariant solution of equation \eqref{eq:kinetic_equation}.
\end{proof}

\begin{proof}[Proof of Proposition \ref{lem:equilibria}]
This proof does not need the technicalities associated with martingales, as the initial distributions of the process are invariant, and every time an interaction event occurs their distribution remains unchanged. The theorem can be proven in a direct way, without even the Poissonisation trick. 

Consider a continuous function $g$ on $[0,1]$ and assume that $\|g\|_{\infty} \le B$. 
Let $\e > 0$ and select a $\delta > 0$ so that $\delta < \e/2 \wedge B$. Furthermore assume  that 
$N$ is large enough so that for a fixed $\beta$,  $0 < \beta < 1$ we have that 
\[
\sup_{x \in [0, N^{-\beta}]}|g(0) - g(x)| < \delta.
\]  
In order to prove the result we just need to show that $\langle g, \mu^N_0 \rangle \to g(0)$ as $N \to \infty$. We will show that this happens $\P$- a.s., when $\P =  \otimes_{N = 2}^{\infty}\mu^{\infty, N}_0$ the product measure on the space $\otimes_{N = 2}^{\infty} \Delta_{N-1}.$ 

We have that $\langle g,  \mu^N_0 \rangle = N^{-1} \sum_{i=1}^N g(X^{N_0}_i)$, so for the $\P-$ a.s.\ convergence we estimate 
\begin{align*}
\P\Big\{ \Big| \frac{1}{N} &\sum_{i =1}^N g(X_i)- g(0)\Big| > \e \Big\} = \P\Big\{ \Big| \sum_{i =1}^N (g(X_i)- g(0))\Big| > N\e \Big\}\\ 
&\le \P\Big\{ \sum_{i =1}^N \big| g(X_i)- g(0)\big| > N\e \Big\}\\ 
&\le e^{-\varepsilon  N} \E\Big(\exp\Big\{ \sum_{i =1}^N \big| g(X_i)- g(0)\big| \Big\}\Big)\\
&= e^{-\varepsilon  N} \E\Big(\exp\Big\{ \sum_{i =1}^N \big| g(X_i)- g(0)\big| \Big\}\sum_{ I \subseteq [N] } \mathbbm1\{ X_i \ge N^{-\beta}, i \in I  \}  \mathbbm1\{ X_i < N^{-\beta}, i \notin I  \}\Big)\\
&=  e^{-\varepsilon  N} \E\Big(\sum_{ I \subseteq [N] } e^{\sum_{i \in I} | g(X_i)- g(0)|}\mathbbm1\{ X_i \ge N^{-\beta}, i \in I  \} e^{\sum_{i \notin I} | g(X_i)- g(0)|} \mathbbm1\{ X_i < N^{-\beta}, i \notin I  \}\Big)\\
&\le  e^{-\varepsilon  N} \E\Big(\sum_{ I \subseteq [N] } e^{2 B|I|}\mathbbm1\{ X_i \ge N^{-\beta}, i \in I  \} e^{(N - |I|) \delta} \mathbbm1\{ X_i < N^{-\beta}, i \notin I  \}\Big)\\
& \le  e^{-\varepsilon  N} \sum_{ k = 0}^N {N \choose k} e^{2 Bk + (N - k) \delta }\E\big(\mathbbm1\{ X_i \ge N^{-\beta} \text{ for $k$ indices}  \}\big).
\end{align*}
The last line has the simplified sum index because of exchangeability of the coordinates, and it is an upper bound, because we dropped the second indicator function. 
Before proceeding with the calculation, we just bound the last expectation when $k$ is not zero. Note that if $k > [N^{1-\beta}]$, the indicator inside is identically zero, otherwise the total wealth cannot be one.  
We also restrict the index of summation to $[N^{1-\beta}]$ as the indicator vanishes otherwise.
\begin{align}
\P\Big\{ \Big| \frac{1}{N} \sum_{i =1}^N g(X_i)- g(0)\Big| > \e \Big\} &\le e^{(\delta-\varepsilon) N} \sum_{ k = 0}^{[N^{1-\beta}]} {N \choose k} e^{(2 B - \delta) k} \notag\\
&\le e^{-\varepsilon N /2} N^{1-\beta} {N \choose [N^{1-\beta}]} e^{(2 B - \delta) N^{1-\beta}}. \label{eq:summable}
\end{align}
The last line follows because eventually $\delta$ will vanish and the exponent $(2B -\delta)$ will be eventually positive. therefore the maximum term in the sum is the last one, when $k = [N^{1-\beta}]$ as combinations are also increasing until around $N /2$. 
Finally, one can use Stirling's formula to see that asymptotically there exists a constant $c$ so that 
\[
 {N \choose [N^{1-\beta}]} \sim e^{cN^{1-\beta}}.
\]
Therefore the upper bound in equation \eqref{eq:summable} is summable over $N$. 
A final application of the Borel-Cantelli lemma completes the proof. 
\end{proof}

In the remaining part of this subsection, we discuss invariant measures  that are absolutely continuous with respect to  the Lebesgue measure on $\R_+$. The blanket assumption  is that for each $t \ge 0$, is that there exists a probability density function $f_t$ so that 
 \[ \mu_t(x) =  f_t(x) \,dx, \] 
and we can find invariant measures with this property. We will show that one family of such measures can be obtained as limits of the empirical measures and (it is therefore a true invariant measure) that the other family cannot and therefore the kinetic equation \eqref{eq:kinetic_equation} does give extraneous solutions.   

The first step is to find the restriction of the operator $Q$ to the class of absolutely continuous measures, which we will call $\bar Q$.
Using $\bar Q$, we can formally write an equation for the evolution of the assumed densities $f_t$. Assume that $\displaystyle \lim_{N \to \infty} W_N = w_0 \in (0, \infty]$. For any value of $w_0$ we will denote the restricted operator by $\bar Q_{w_0}$, and $\bar Q_{w_0}$ acts on probability densities $f$ on $\R_+$. In other words, for any $g \in C_b(\R_+)$
\[\langle g, \, Q(\mu) \rangle = \langle g,\, \bar Q_{w_0}(f) \rangle, \quad \text{ whenever } \mu (x) = f(x) \, dx. \]
First notice that when the measure $\mu$ has a density $f$ we can write
\begin{align*}
\int_0^1\int_{\R^+}\int_{\R^+} g((1-r)(x+y)) &\mathbbm1_{\{x+y\leq w_0\}}f(x)f(y)dx\,dy\,dr \\
&=\int_0^1\int_{\R^+}\int_{\R^+}  g(r(x+y)) \mathbbm1_{\{x+y\leq w_0\}}f(x)f(y) dx\,dy\,dr
\end{align*}
with a change of variables $r \mapsto 1-r$. Therefore, expression \eqref{eq:operator_Q} can we rewritten as
\be \label{eq:rest}
\la g,  \,Q(\mu) \ra = \int_{[0,1]}\int_{\R_+^2}[2 g(r(x+y))-  g(x)-g(y)]\mathbbm1_{\{x+y\leq w_0\}}f(x)f(y)dx\,dy\,dr.
\ee
Now it follows that 
\begin{align*}
\int_{[0,1]}\int_{\R_+^2} g(r(x+y))&\mathbbm1_{\{x+y\leq w_0\}}  f(x) f(y) dx\,dy\,dr \\
 &=\int_{[0,1]}\int_{\R_+^2} g(u+p)\mathbbm1_{\{u+p\leq r w_0\}} f(u/r) f(p/r) du\, dp\,\frac{dr}{r^2}\\
&=\int_{[0,1]}\int_{\R^2_+} g(z) \mathbbm1_{\{z\geq p\}}\mathbbm1_{\{z\leq r w_0\}} f((z-p)/r) f(p/r) dz\,dp\frac{dr}{r^2}\\
&= \int_{[0,1]}\int_{\R^2_+} g(x) \mathbbm1_{\{x\geq y\}} \mathbbm1_{\{x\leq r w_0\}}f((x-y)/r) f(y/r) dx\,dy\frac{dr}{r^2}
\end{align*}
where in the first equality we made the change of variables $rx=u$, $ry=p$; in the second equality we made the change of variables $z=u+p$; in the last equality we just changed the name of the labels $z=x$, $p=y$. With similar computations, we obtain that
\[ \int_{\R_+}\int_{\R_+}g(x) \mathbbm1_{\{x+y\leq w_0\}} f(x) f(y) dx\,dy= \int_{\R_+} g(x) f(x) \lp \int_{0}^{(w_0-x)_+} f(y) dy \rp dx, \]
where $(w_0-x)_+= (w_0-x)\mathbbm1_{(w_0-x)\geq 0}$.

Combine these calculations into \eqref{eq:rest} to obtain that $\bar Q_{w_0}$ is given by 
\be \label{eq:kinetic_Q}
\bar Q_{w_0}(f):= 2 \int^{1}_{x/w_0} \int_0^x f\lp \frac{y}{r}\rp f\lp\frac{x-y}{r}\rp\, dy \frac{dr}{r^2}- 2 f(x) \int_0^{(w_0-x)_+}\!\!\!\!\!\!\!\!\!\!\!\!f(y) dy,
\ee
Similarly, the evolution of the density functions can be obtained (in a weak sense) from 
\beqarl
\label{eq:kinetic_f}
f_t &=& f_0 + \int^t_0 \bar Q_{w_0}(f_s) \ ds.
\eeqarl
Note that when $w_0<\infty$ and $x>w_0$, then $\bar Q_{w_0}(f)(x)=0$.
When $w_0=\infty$ the operator $\bar Q_\infty$ reads
	\be \label{eq:Qinf}
	\bar Q_\infty(f) = 2 \int^1_0 \int^x_0 f\lp \frac{y}{r}\rp f\lp \frac{x-y}{r}\rp dy \frac{dr}{r^2}- 2f(x),
	\ee
since $f(x)$ is a probability density.
In order to prove Corollary \ref{cor:kinetic_and_equilibria}, it suffices to show that the proposed equilibria annihilate $\bar Q$. We re-state the corollary, using this observation.

\begin{corollary}
Let $\displaystyle \lim_{N \to \infty} W_N = w_0 \in (0, \infty]$. Assume that $\mu_t$ is a solution of \eqref{eq:kinetic_equation} which has has a density $f_t$ for all $t$,  
satisfying equation \eqref{eq:kinetic_f}
Then 
\begin{enumerate}
\item If $w_0 = \infty$, the exponential distributions 
\be 
\label{expagain}
\tilde f(x) = \frac{e^{-x/m}}{m},
\ee
are equilibria for the operator $\bar Q_{\infty}$, (i.e. $\bar Q_{\infty}(\tilde f) = 0$) and remain invariant under \eqref{eq:kinetic_f}. In particular, if $f_0$ is of the form \eqref{eq:equilibria} with
$\la x, f_0 \ra = m_0>0,$ then the distribution \eqref{eq:equilibria} with $m= m_0$ is a stationary solution of \eqref{eq:kinetic_f}. 

\item If $0<w_0<\infty$, then the following distributions are compactly supported in $[0,w_0]$ and are equilibria for the operator $\bar Q_{w_0}$
\be 
\tilde f(x)  = \frac{e^{-x/m}}{m(1-e^{-w_0/m})} \mathbbm1_{\{ x\leq w_0\}}.
\ee
\item (Uniqueness of the invariant family at $w_ 0 = \infty$) Moreover, under the extra assumption that the density $f_t$ is differentiable on $\R_+$, then measures with density \eqref{eq:equilibria} or, equivalently, \eqref{expagain} are the unique equilibria of $\bar Q_{\infty}$. 
\end{enumerate}
\end{corollary}

\begin{proof}[Proof of Corollary \ref{cor:kinetic_and_equilibria}]
It is straightforward to check that $\bar Q_{w_0}(\tilde f) =0$ for both $w_0 = \infty$ and $w_0 < \infty$. Also, if 
$f_0 = \tilde f$ 
with $\la x, f_0\ra =m_0$ this implies that $f_0$ is stationary solution of \eqref{eq:kinetic_f} with $m=m_0$. 
It remains to show item $(3)$. Select any invariant $f$ and for that, recall that $\bar Q_\infty(f) = 0$. Let $X, Y$ be independently distributed with density $f$, and $U$ a uniform r.v. on $[0,1]$. Start from equation \eqref{eq:rest}, and observe that a different way to write it is  
\[
0 =\langle g, \bar Q_{\infty}(f) \rangle= \langle g, Q(\mu) \rangle =2 \E_{(U, X, Y)}[ g(U(X+Y))] - 2  \E_{X}[g(X)],
\]
and therefore, the distribution of $U(X+Y)$ is the same as the distribution of $X$. If we now condition on the value of $X+ Y := S = s$, we have that the conditional distribution of $X$ given $S = s$ is that of a uniform r.v.~ on $[0, s]$. Let $f_S$ denote the density of the sum $X+Y$ and $f_{X|S}$ the conditional density of $X$ given $S =s$. We can write 
\begin{align*}
f(x) = \int_{x}^{\infty} f_{X|S}(x|s) f_{S}(s) \, ds =  \int_{x}^{\infty} \frac{1}{s} f_{S}(s) \, ds.
\end{align*}
Now use the fundamental theorem of calculus to differentiate both sides with respect to $x$ in order to obtain 
\[
f'(x) = - \frac{1}{x}f_S(x) \Longleftrightarrow  x f'(x) = - f_S(x).
\]
Take the Laplace transform of the equation above; denote by $\bar g (t) $ the Laplace transform of $g(x)$ and use basic properties on the equation in the last display, to argue that 
\[
LHS = \overline{ x f'(x)} = - \frac{d}{dt} \overline{f'(x)}(t) = -  \frac{d}{dt} (t \bar f (t) )= - \bar f  - t \frac{d \bar f}{ dt},
\] 
while the Laplace transform of the convolution that gives the density of $S$ is
\[
RHS = -(\bar f)^2.
\] 
These give rise to the differential equation 
\[
\frac{d \bar f}{\bar f( \bar f-1)} = \frac{dt}{t}.
\]
The solution to the differential equation,  for some constant $m$, is 
\[
\log \left| \frac{\bar f -1}{\bar f } \right| = \log mt.
\]
Keep in mind that since $t >0$ and $\bar f(t) < 1$, we can solve 
\[ \bar f(t)= \frac{m^{-1}}{ m^{-1} + t},\]
where we identify the Laplace transform of an exponential distribution with mean $m$.
\end{proof}

\begin{proof}[Proof of Proposition \ref{prop:2am}]
Here is the proof of the two points.
\begin{enumerate}
	\item We only need to show the convergence of the initial measures. Consider a function $g \in C_b(\R_+)$ and compute
		\begin{align*}
			\langle g, \mu^{N}_0 \rangle = \frac{1}{N} \sum_{i=1}^N g(X_i) \longrightarrow \E_{\text{Exp}(1/ m_0)}g(X_1) = \int_{0}^{\infty} g(x) \mu_0(dx),
		\end{align*}
		by the law of large numbers. This verifies the definition of weak convergence $\mu_0^N \Longrightarrow \mu_0$.  
	\item Assume the contrary, and consider a sequence of converging initial measures. Since $w_0 < \infty $, Proposition \ref{cor:delta} gives that $\mu_0^N$ should converge to $\delta_0$, which does not have a density  \eqref{eq:equilibria_bounded}. This gives the desired contradiction. \qedhere
\end{enumerate}
\end{proof}

\section{An application to partitions of integers}
\label{sec:partitions} 

The coagulation-fragmentation process is very versatile and therefore is well-studied and it can be viewed also as a process on integer partitions of integers. To be precise,
for any fixed $N \in \N$ we have that $ \sum_{i} X^{i, N}_t = W_N$. If we assume $W_N$ is an integer, we can  interpret the vector ${\bf X}^N_t$ as a random (real) partition of the integer $W_N$ and the process $\{ {\bf X}^N_t\}_{t \ge 0}$ can be viewed as a Markov chain on these partitions. Most recently, a version of the process (with deterministic binary interactions at discrete time steps) has been studied in \cite{Cha-Dia-19} in terms of its rate of convergence to the equilibrium. 

In this section we cast the results of the previous sections in terms of partitions. Conditions on the value $W_N$, or on the nature of partitions will differ based on the application. In the process we generalise or recover theorems proven in \cite{vershik2003asymptotics} about various scalings of uniform integer partitions. 

First, we discuss the case where we want only integer partitions to have mass in our process. 

\subsection{ Integer partitions, $W_N = n$ fixed, $N$ fixed.}  

In the DS-DT model, we have the process ${\bf X}_t^{n, N}$ which describes the evolution of the partition process on $\Delta_{N-1}(n)$. The state space is given by \eqref{eq:D_Nn} and therefore the process $n {\bf X}_t^{n, N}$ has state space all $N$-term integer partitions of the number $n$. 

Since the transition matrix given by  \eqref{transition} is doubly stochastic, letting time $t$ to infinity, we have that the limiting measure on these partitions is the uniform measure on $n\Delta_{N-1}(n)$, as the chain is also irreducible and aperiodic. 

\subsection{Uniform measure on integer partitions, $W_N/N \to c \in [0,\infty] $, $N \to \infty$.}  

Define a uniform initial measure on $W_N\Delta_{N-1}(W_N)$. There are many ways to construct it and we choose the following:

Consider an i.i.d.\ sequence of geometric random variables  $\{ G_i(p)\}_{i \in \N}$ with mass function
\be \label{eq:g}
\P\{  G_i = k \} = p(1- p)^k, \quad k = 0, 1,2, \ldots 
\ee
Let ${\bf u }= (u_1, \ldots u_N) \in W_N\Delta_{N-1}(W_N) $. Then sample according to the conditional measure 
\[ 
\nu^N_0({\bf u}) = P\Big\{ (G_1, \ldots, G_N ) = (u_1, \ldots, u_N) \Big| \sum_{i=1}^N G_i = W_N\Big\}.
\]
The measure $\nu^N_0$ is uniformly distributed on $W_N\Delta_{N-1}(W_N)$. This fact is irrespective of the value of the parameter $p$ and irrespective of the value of $W_N$. Assume $W_N \to \infty$. 
There are three cases to consider, based on $\lim_{N \to \infty}\frac{W_N}{N}$.   

\subsubsection{ The case $\displaystyle \limsup_{N \to \infty}\frac{W_N}{N} = 0$}

In this case,  for any $\e> 0$, and for $N$ large enough $\frac{W_N}{N} < \e$.  Let $(X_1, \ldots, X_N) \sim \nu^N_0$ and define $\mu^N_0 = \frac{1}{N} \sum_{i=1}^N \delta_{X_i}$ supported on $\R_+$. 
%
The proof of Proposition \ref{lem:equilibria} can be repeated with minor modifications, and the sequence 
$\{\mu_0^N\}_N$ converges weakly to $\delta_0$. Theorem \ref{th:hydrodynamic_limit} can now be applied.

\subsubsection{The case $\displaystyle \liminf_{N\to\infty} \frac{W_N}{N} = \infty$}
 As we mentioned after the statement of Theorem \ref{th:hydrodynamic_limit}, when the wealth grows superlinearly, the theorem does not necessarily apply. However there is a way to scale using a random approximation to $W_N$ so that the theorem works. 

In this case we fix 
\[
p_N = \frac{N}{W_N},
\]
as the success probability of each independent geometric (so the sequence refreshes with every $N$). To denote this dependence we write $G^{(N)}_i$ for each geometric. This does not alter the fact that the conditional distribution is uniform on the integer simplex $W_N\Delta_{N-1}(W_N)$. Let $\P_N$ denote the law of these geometrics. Then define 
\be \label{eq:geom}
(X_1^N, \ldots, X^W_N) = 
\begin{cases}
(1, \ldots, 1), & {\rm when } \sum_{i=1}^N G^{(N)}_i = 0,\\
N \Big( \frac{G^{(N)}_1}{\sum_{i=1}^N G^{(N)}_i}, \ldots, \frac{G^{(N)}_N}{\sum_{i=1}^N G^{(N)}_i} \Big), & {\rm otherwise.}
\end{cases}
\ee
This is a distribution on the simplex $N \Delta_{N-1}$. Conditional on the value of $\sum_{i=1}^N G^{(N)}_i = w_N$, this distribution is uniform on the discrete simplex with mesh $N/w_N$. 

\begin{lemma}\label{lem:asGn} Consider a triangular array with independent rows indexed by $N$ and random entries $\{G_i^{(N)}\}_{1\le i \le N, N \in \N}$. Each row
$
\{G_i^N\}_{1 \le i \le N} 
$
consists of i.i.d.\ geometric random variables with success probability $p_N = N/W_N$, so that $\lim_{N\to \infty}p_N = 0$. Let $\P$ denote  the law of the array. Then 
\[ \frac{\sum_{i = 1}^N G_i^{(N)}}{W_N} = 1, \quad \P-{\rm a.s.}\]
Moreover, for any $\beta < 1/2$, we can find a constant $c_{\beta}$ such that 
\be \label{eq:moder}
\P_N\Big\{ \Big| \frac{ \sum_{i = 1}^N G_i^{(N)}}{W_N} - 1 \Big|> N^{-\beta} \Big\} \le e^{-c_\beta N^{1-2\beta}}
\ee
\end{lemma}

\begin{proof}
Let $\P_N$ the marginal of the $N$-th row. Let $\e > 0$ and for $0< t < -\log(1-p_N)$, use a Chernoff bound 
\begin{align*}
\P_N\Big\{ \frac{ \sum_{i = 1}^N G_i^{(N)}}{W_N} > 1+\e \Big\} &\le e^{-t W_N (1+\e)}\E_{\P_N}(e^{t\sum_{i = 1}^N G_i^{(N)}})=e^{-t W_N (1+\e)}\Big(\E_{\P_N}(e^{tG_i^{(N)}}) \Big)^N\\
&= e^{-t W_N (1+\e)}\Big(\frac{p_N}{1-(1-p_N)e^t}\Big)^N.
\end{align*}
Let $\alpha \in (0,1)$ and let  $t =  -\alpha \log(1-p_N)$ for $N$ large enough and we have that 
\[
\P_N\Big\{ \frac{ \sum_{i = 1}^N G_i^{(N)}}{W_N} > 1+\e \Big\} \le  e^{ \alpha W_N (1+\e)\log(1-p_N)}\Big(\frac{p_N}{1-(1-p_N)^{1-\alpha}}\Big)^N.
\]
Note that for $x > 0$ small enough we have $(1 - x)^{1-\alpha} < 1 - (1-\alpha)x$ and $\log(1-x)< -x$. For $x = p_N$ we further bound 
\be\label{eq:mod}
\P_N\Big\{ \frac{ \sum_{i = 1}^N G_i^{(N)}}{W_N} > 1+\e \Big\} \le  e^{ -\alpha N (1+\e)}\Big(\frac{1}{{1-\alpha}}\Big)^N = e^{(-\alpha(1+\e) - \log(1-\alpha))N}.
\ee
The function $-\alpha(1+\e) - \log(1-\alpha)$ attains a minimum when $\alpha = \frac{\e}{1+\e}$ and at that point the value is negative. Therefore we found a constant $c_1(\e)$ so that 
\be\label{eq:1G}
\P_N\Big\{ \frac{ \sum_{i = 1}^N G_i^{(N)}}{W_N} > 1+\e \Big\} \le  e^{-c_1(\e)N}.
\ee
Similarly, we can find a constant $c_2$ such that  
\be \label{eq:2G}
\P_N\Big\{ \frac{ \sum_{i = 1}^N G_i^{(N)}}{W_N} < 1-\e \Big\} \le  e^{-c_2(\e)N},
\ee
by first multiplying in the probability with $t < 0$. The almost sure convergence follows from the Borel-Cantelli lemma.

For the second part of the Lemma return to equation \eqref{eq:mod} and set $\e = \e_N = N^{-\beta}$ for any $\beta < 1/2$. Then $\alpha = N^{-\beta}/(1+ N^{-\beta})$. A Taylor expansion on $\log(1+\e_N)$ now gives 
\[
\P_N\Big\{ \frac{ \sum_{i = 1}^N G_i^{(N)}}{W_N} > 1+ N^{-\beta} \Big\} \le e^{(- \e_N + \log(1+\e_N))N} \le e^{-c_\beta N^{1-2\beta}}. 
\]
The lower bound follows in a similar way, using \eqref{eq:2G}.
\end{proof}

The following lemma is a direct consequence of the strong law of large numbers for triangular arrays \cite{taylor1987strong} so we omit the proof. The fact that the limiting law is that  of an exponential follows from the fact that for each finite collection $\mathcal I$ of $i$, the vector $\Big(\frac{N}{W_N}G_i^{(N)}\Big)_{i \in \mathcal I}$ converges weakly to a vector of independent exponential random variables of rate 1.

\begin{lemma}(SLNN for the triangular array) \label{lem:LLNar}Consider a triangular array with independent rows indexed by $N$ and random entries $\{G_i^{(N)}\}_{1\le i \le N, N \in \N}$. Each row
$
\{G_i^N\}_{1 \le i \le N} 
$
consists of i.i.d.\ geometric random variables with success probability $p_N = N/W_N$, so that $\lim_{N\to \infty}p_N = 0$. Let $\P$ denote  the law of the array. Let  $g \in C_b(\R_+)$ and denote by 
\[ \E_{\rm Exp(1)}(g) = \int_{0, \infty} g(x)e^{-x}\,dx.\]
Then 
\[
\frac{1}{N}\sum_{i=1}^N g\Big( \frac{N}{W_N}G_i^{(N)}\Big) \to  \E_{\rm Exp(1)}(g), \quad \P-\rm{a.s.}
\]
\end{lemma}

An immediate consequence of the strong law is that for any $\e>0$ and any fixed $g \in C_b(\R_+)$, 
\be\label{eq:summablear}
\sum_{N=1}^\infty \P\Big\{ \Big|\frac{1}{N}\sum_{i=1}^N g\Big( \frac{N}{W_N}G_i^{(N)} \Big) - \E_{\rm Exp(1)}(g)\Big| > \e  \Big\} < \infty.
\ee
If not, since the rows of the array are independent, the second Borel-Cantelli lemma would give that for a.e.\ realisation, convergence is not possible which would lead to a contradiction.

Now, for any integer $K>1$, we estimate the number of coordinates in each row of the array, with value that exceeds $K$. To this effect, define auxiliary Bernoulli variables 
\[
\mathcal B_i^{N, K} =
\begin{cases}
	1, &{ \rm if \,} \frac{ N G^{(N)}_i }{W_N} > K,\\
	0, &{\rm otherwise}.
\end{cases}
\] 
Then
\begin{align}
\notag \P\Big\{ \sum_{i=1}^N \mathcal B_i^{N, K} > N K^{-1/2}  \Big\} &\le \P\Big\{ \frac{N}{W_N}\sum_{i=1}^N G_i^{(N)} >  NK^{1/2}  \Big\} \le \P\Big\{ \frac{1}{W_N}\sum_{i=1}^N G_i^{(N)} > K^{1/2}  \Big\} \\
&\le  e^{-c_K N}, \quad {\rm by \eqref{eq:1G}.} \label{eq:Bi}
\end{align}

\begin{theorem}\label{thm:partG}
Let $(X_1^N, \ldots, X^N_N)$ be distributed as in \eqref{eq:geom}, and define the empirical measure $\mu_0^N = \frac{1}{N} \sum_{i=1}^N \delta_{X_i^N}$. Then 
\[
\mu_0^N \Longrightarrow \mu_{0} \sim {\rm Exp(1)}, \quad \P-a.s.
\]
As such, for all $t > 0$ we have that the mean-field limit will satisfy $\mu_{t} \sim {\rm Exp(1)}$.
\end{theorem}

\begin{remark} Note that our initial sequence of measures is not uniform on the sequence of simplices $N \Delta_N$, in contrast with Theorem 2 and Corollary 2 of \cite{vershik2003asymptotics}. However, in light of Lemma \ref{lem:asGn}, the initial measures can be viewed as approximation of discrete uniform measures on mesh $1/W_N$ for $\Delta_N$. Or, one can view them as measures on partitions of the random number $\sum_{i=1}^N G^{(N)}_i$. 

Finally, one can obtain similar statements as Theorem 2 and Corollary 2 from \cite{vershik2003asymptotics}, by using appropriate bounded continuous functions $g_1, g_2$ in the duality relation. Namely $g_1$ needs to be a continuous approximation of $\mathbbm{1}_{[t, \infty)}$ and $g_2$ a continuous approximation of  $ x \mathbbm1_{[0,t]}(x) $. 
\end{remark}

\begin{proof}[Proof of Theorem \ref{thm:partG}]
Fix an $\e > 0$, $\beta \in (0,1)$ and and $K \in \N$. Also fix $g \in C_b(\R_+)$. We define the events 
\[
A_{N, \beta} = \Big\{ \Big|  \frac{ \sum_{i = 1}^N G_i^{(N)}}{W_N} -1  \Big| > N^{-\beta} \Big\}, \quad  B_{N, K} = \Big\{ \sum_{i=1}^N \mathcal B_i^{N, K} > K^{-1/2} N \Big\}, 
\]
and 
\[
C_{N, g, \e} = \Big\{ \Big|\frac{1}{N}\sum_{i=1}^N g\Big( \frac{N}{W_N}G_i^{(N)} \Big) - \E_{\rm Exp(1)}(g)\Big| > \e/4  \Big\}.
\]
Equations \eqref{eq:moder}, \eqref{eq:Bi} and \eqref{eq:summablear}  imply that 
\be\label{eq:allsum}
\sum_{N=1}^{\infty} \P\{ A_{N, \beta} \cup B_{N, K} \cup C_{N, g, \e}\} < \infty.
\ee

We perform the following estimates on $A_{N, \beta}^c \cap B_{N, K}^c \cap C^c_{N, g, \e}$. In particular, since we are on $A_{N, \beta}^c$ we have that $ \sum_{i = 1}^N G_i^{(N)} \neq 0$. Define  
\[ I_{N, K} = \Big\{ i: \frac{NG_i^{(N)}}{W_N} \le K \Big\}.\] 
Since we are on $B_{N, K}^c$, we have that  $(1 - K^{-1/2}) N \le |I_{N, K}| \le N$.

Then we write 
\begin{align}
\notag \langle g, \mu_0^N \rangle 
\notag &= \frac{1}{N} \sum_{i =1}^N g(X_i^N) =  \frac{1}{N} \sum_{i =1}^N g\left(\frac{N G_i^{(N)}}{ \sum_{i = 1}^N G_i^{(N)}} \right) = \frac{1}{N} \sum_{i =1}^N g\left(\frac{N G_i^{(N)}}{W_N}\frac{W_N}{ \sum_{i = 1}^N G_i^{(N)}} \right) \\
\notag &=\frac{1}{N} \sum_{i =1}^N g\left(\frac{N G_i^{(N)}}{W_N}\frac{W_N}{ \sum_{i = 1}^N G_i^{(N)}} \right) \mp \frac{1}{N} \sum_{i =1}^N g\left(\frac{N G_i^{(N)}}{W_N} \right)\\
\notag &\le  \frac{1}{N} \sum_{i =1}^N g\left(\frac{N G_i^{(N)}}{W_N} \right) + \frac{1}{N}\sum_{i=1}^{N} \left| g\left(\frac{N G_i^{(N)}}{W_N}\frac{W_N}{ \sum_{i = 1}^N G_i^{(N)}} \right)- g\left(\frac{N G_i^{(N)}}{W_N} \right) \right|\\
\label{eq:step1}&=  \frac{1}{N} \sum_{i =1}^N g\left(\frac{N G_i^{(N)}}{W_N} \right) + \frac{1}{N}\sum_{i \in I_{N, K}} \left| g\left(\frac{N G_i^{(N)}}{W_N}\frac{W_N}{ \sum_{i = 1}^N G_i^{(N)}} \right)- g\left(\frac{N G_i^{(N)}}{W_N} \right) \right|\\
\notag &\phantom{XXXXxxxxxxxxXXX}+   \frac{1}{N}\sum_{i \notin I_{N, K}} \left| g\left(\frac{N G_i^{(N)}}{W_N}\frac{W_N}{ \sum_{i = 1}^N G_i^{(N)}} \right)- g\left(\frac{N G_i^{(N)}}{W_N} \right) \right|.
\end{align}

Since $g$ is a bounded continuous function, its restriction on $[0,2K]$ is uniformly continuous and admits a modulus of continuity such that 
\[
|g(x) - g(y)| = \omega_{g, K}(|x -y|), \quad \textrm{ for all $x, y$ in $[0,2K]$.}
\]
Now assume $i \in I_{N,K}$. Since we are on $A^c_{N, \beta}$, we have that  
\[
\frac{N G_i^{(N)}}{W_N}\frac{1}{1+N^{-\beta}}\le \frac{N G_i^{(N)}}{W_N}\frac{W_N}{ \sum_{i = 1}^N G_i^{(N)}}\le \frac{N G_i^{(N)}}{W_N}\frac{1}{1-N^{-\beta}}. 
\]
This implies that 
\[
\left| g\left(\frac{N G_i^{(N)}}{W_N}\frac{W_N}{ \sum_{i = 1}^N G_i^{(N)}} \right)- g\left(\frac{N G_i^{(N)}}{W_N} \right) \right| \le \omega_{g, K}\left(K\frac{4}{N^{\beta}} \right),
\]
therefore, from \eqref{eq:step1} 
\begin{align*}
\langle g, \mu_0^N \rangle &\le \frac{1}{N} \sum_{i =1}^N g\left(\frac{N G_i^{(N)}}{W_N} \right) +  \omega_{g, K}\left(\frac{4K}{N^{\beta}}\right) \frac{|I_{N, K}|}{N} + 2 \|g\|_{\infty}  \frac{N - |I_{N, K}|}{N}  \\
&\le \frac{1}{N} \sum_{i =1}^N g\left(\frac{N G_i^{(N)}}{W_N} \right) +  \omega_{g, K}\left(\frac{4K}{N^{\beta}}\right)+ 2 \|g\|_{\infty} \frac{1}{\sqrt K} \\
&\le \E_{\rm Exp(1)}(g) + \frac{\e}{4} + \omega_{g, K}\Big(\frac{4K}{N^{\beta}}\Big)+ 2 \|g\|_{\infty} \frac{1}{\sqrt K}.
\end{align*}
The last inequality follows from the fact that we are on $C_{N, g, \e}^c$. The symmetric lower bound 
\[
\langle g, \mu_0^N \rangle \ge \E_{\rm Exp(1)}(g) - \frac{\e}{4} - \omega_{g, K}\left(\frac{4K}{N^{\beta}}\right)- 2 \|g\|_{\infty} \frac{1}{\sqrt K},
\]
is obtained in an identical manner.

Now, for a fixed $\e>0$ 
\begin{enumerate}
\item Select $K = K(\e, g)$ large so that $2\|g\|_{\infty} < \frac{\e}{4}\sqrt{K}$.
\item With $K$ fixed, select$N= N(\e, g, \beta)$ large enough so that $ \omega_{g, K}\Big(\frac{4K}{N^{\beta}}\Big) < \e /4$.
\end{enumerate}
These choices give  
\[
|\langle g, \mu_0^N \rangle - \E_{\rm Exp(1)}(g)| < \frac{3}{4}\e < \e.
\]
Then, we have just shown the inclusion of events
\[
\{ |\langle g, \mu_0^N \rangle - \E_{\rm Exp(1)}(g)| > \e\} \subseteq  \{A_{N, \beta} \cup B_{N, K} \cup C_{N, g, \e}\},
\]
for all $N$ large enough.
By \eqref{eq:allsum}and an application of the first Borel-Cantelli lemma, we have that for any fixed $\e>0$, 
\[
\lim_{N \to \infty} |\langle g, \mu_0^N \rangle - \E_{\rm Exp(1)}(g)| \le \e,  \quad \P-\textrm{a.s.}
\]  
Let $\e \to 0$ on a countable sequence, to obtain the result.
\end{proof}

\subsubsection{The case $\displaystyle \varliminf_{N\to\infty} \frac{W_N}{N}  \in(0, \infty)$} 
Usually when we have partitions, one may want the number of terms in the partition of a number to be less than the number itself. This restriction does not apply in this example.
We left this case for last, as the initial weak convergence does not directly lead to invariant measures that we discussed. The previous proofs work directly for this case as well if we choose our parameters correctly.  
 
Start with i.i.d.\ Geom($p$) random variables  $\{ V^N_i \}_{1\le i \le N, n \in \N }$ with mass function given by \eqref{eq:g}. We do not want any degeneracies so $p \in (0,1)$ fixed. Consider the vector 
\be\label{eq:exp31}
(X_1^N, \ldots, X_N^N) = \left( \frac{(1-p)NU_1^N}{p\sum_{i=1}^N U_i^N}, \ldots, \frac{(1-p)NU_N^N}{p\sum_{i=1}^N  U_i^N} \right) \in \frac{1-p}{p}N\Delta_{N-1}.
\ee
Then $W_N/N = \frac{1-p}{p} \in (0, \infty)$ and depending on the value of $p$ the limit can be any positive number. Then the methodology of the previous subsection applies with virtually no changes, except the convergence in Lemma \ref{lem:LLNar} changes to the expectation of a geometric random variable, since $p$ is now fixed. We leave the details to the interested reader. 

\begin{theorem}\label{thm:partD}
Let $(X_1^N, \ldots, X^N_N)$ be distributed as in \eqref{eq:exp31}, and define the empirical measure $\mu_0^N = \frac{1}{N} \sum_{i=1}^N \delta_{X_i^N}$. Then 
\[
\mu_0^N \Longrightarrow \mu_{0} \sim {{\rm Geom}(p)}, \quad \P-a.s.
\]
\end{theorem}

\begin{remark} Contrast this with Theorem 3 of \cite{vershik2003asymptotics}. The limit there is also a geometric random variable, albeit one that is supported on $\N$. Our measure on partitions here is not uniform on the $N$-partitions of $\frac{1-p}{p}N$ however, just a convenient approximation of it. It would be of interest to study the equilibrium measure for the kinetic equation, when we are starting from a geometric initial measure. 
\end{remark}


\subsection{ Real partitions} The development here is identical (but easier) as in the discrete case. We omit most details given that the previous proofs can be repeated. 

Start with i.i.d.\ Exponential(1) random variables  $\{ U^N_i \}_{1\le i \le N, n \in \N }$ and consider the vector 
\be\label{eq:exp3}
(X_1^N, \ldots, X_N^N) = \left( \frac{NU_1^N}{\sum_{i=1}^N U_i^N}, \ldots,  \frac{NU_N^N}{\sum_{i=1}^N  U_i^N} \right) \in N\Delta_{N-1}.
\ee
 Then, the distribution $\nu_0^N$ of $(X_1^N, \ldots, X_N^N)$ is that of a uniform vector on $N\Delta_{N-1}$. 
 
\begin{theorem}\label{thm:partE}
Let $(X_1^N, \ldots, X^N_N)$ be distributed as in \eqref{eq:exp3}, and define the empirical measure $\mu_0^N = \frac{1}{N} \sum_{i=1}^N \delta_{X_i^N}$. Then 
\[
\mu_0^N \Longrightarrow \mu_{0} \sim {\rm Exp(1)}, \quad \P-a.s.
\]
As such, for all $t > 0$ we have that the mean-field limit will satisfy $\mu_{t} \sim {\rm Exp(1)}$.
\end{theorem}

\begin{remark} This implies Theorem 1 and Corollary 1 of \cite{vershik2003asymptotics}. 
\end{remark}

\appendix
\section{Some properties of the Skorokhod space}

\begin{theorem}[Prohorov's theorem (\cite{ethier2009markov}), Chapter 3]
Let $(S,d)$ be complete and separable, and let $\mathcal{M} \in \mathcal{M}_1(S)$. Then the following are equivalent:
\begin{enumerate}
	\item $\mathcal{M}$ is tight.
	\item For each $\e>0$, there exists a compact $K\in S$ such that
	$$\inf_{P\in \mathcal{M}}P(K^\e) \geq 1-\e$$
	where $K^\e:= \{x\in S: \inf_{y\in K} d(x,y)<\e\}$.
	\item $\mathcal{M}$ is relatively compact.
\end{enumerate}
\end{theorem}

Let $(E,r)$ be a metric space. The space $D([0,\infty); E)$ of c\`adl\`ag functions taking values in $E$ is widely used in stochastic processes. In general we would like to study the convergence of measures on this space, however, most of the tools known for convergence of measures are for measures in $\mathcal{M}_1(S)$ for $S$ a complete separable metric space. Therefore, it would be very useful to find a topology in $D([0,\infty); E)$ such that it is a complete and separable metric space. This can be done when $E$ is also complete and separable; and the metric considered is the Skorokhod one. This is why in this case the space of c\`adl\`ag functions is called Skorokhod space. 

Some important properties of this space are the following:
\begin{proposition}[\cite{ethier2009markov}, Chapter 3]
If $x\in D([0,\infty); E)$, then $x$ has at most countably many points of discontinuity.
\end{proposition}

\begin{theorem}[\cite{ethier2009markov}, Chapter 3]
If $E$ is separable, then $D([0,\infty); E)$ is separable. If $(E,r)$ is complete, then $(D([0,\infty);E), d)$ is complete, where $d$ is the Skorokhod metric. 

\end{theorem}

\begin{theorem}
The Skorokhod space is a complete separable metric space.
\end{theorem}

\begin{theorem}[The a.s. Skorokhod representation theorem, \cite{ethier2009markov}, Theorem 1.8, Chapter 3]
\label{th:Skorokhod_representation_theorem}
Let  $(S,d)$ be a separable metric space. Suppose $P_n$, $n=1,2,\hdots$ and $P$  in $\mathcal{M}_1(S)$ satisfy $\displaystyle \lim_{n\rightarrow\infty}\rho(P_n, P)=0$ where $\rho$ is the metric in $\mathcal{M}_1(S)$. Then there exists a probability space $(\Omega, \mathcal{F}, \nu)$ on which are defined $S$- valued random variable $X_n$, $n=1,2, \hdots$ and $X$ with distributions $P_n$, $n=1,2,\hdots$ and $P$, respectively such that $\displaystyle\lim_{n\rightarrow \infty} X_n=X$ almost surely.
\end{theorem}

\begin{theorem}[Tightness criteria for measures on the Skorokhod space] 
See \cite{jakubowski1986skorokhod} Theorem 3.1+ \cite{Billingsley} Theorem 4.2 (1968)]
\label{th:jakubowski_criteria}
Let $(S,\mathcal{T})$ be a completely regular topological space with metrisable compact sets. Let $\mathbb{G}$ be a family of continuous functions on $S$ taking values in $\mathbb{R}$. Suppose that $\mathbb{G}$ separates points in $S$ and that it is closed under addition. Then a family $\{ \mathcal{L}^n\}_{n\in \NN}$ of probability measures in $\mathcal{M}_1(D([0,\infty);S))$ is tight iff  the two following conditions hold:
\begin{itemize}
	\item[(i)] For each $\e>0$ there is a compact set $K_\e \subset S$ such that
	$$\mathcal{L}^n(D([0,\infty);K_\e))>1-\e, \quad n\in \NN.$$
	\item[(ii)] The family $\{ \mathcal{L}^n\}_{n\in\NN}$ is $\mathbb{G}$-weakly tight, i.e., for any $g\in \mathbb{G}$ the family $\{\mathcal{L}^n~\circ~(\tilde g)^{-1}\}_{n\in \mathbb{N}}$ of probability measures on $D([0,\infty);\mathbb{R})$ is tight; where $\tilde g$ is defined as follows:
	$$\tilde g: D([0,\infty); S) \to D([0,\infty); \mathbb{R})$$
	with $[\tilde g(\nu)](t)=g(\nu(t))$ for $\nu\in D([0,\infty); S)$ (so that $\nu(t)\in S$).
\end{itemize}
\end{theorem}

\begin{remark} \cite{jakubowski1986skorokhod} only states the results when the time index is in $[0,1]$ (i.e. in a compact set) and the space is $D([0,1]; S)$. However, when the sequence of measures is tight, the result of \cite{Billingsley} allows the result to generalise when the space is $D([0,\infty); S)$.
\end{remark}


\begin{theorem}[Criteria for tightness in Skorokhod spaces (\cite{ethier2009markov}, Corollary 7.4, Chapter 3)]
\label{th:criteria_tightness_Skorokhod}
Let $(E,r)$ be a complete and separable metric space, and let $\{X_n\}$ be a family of processes with sample paths in $D([0,\infty); E)$. Then $\{X_n\}$ is relatively compact iff the two following conditions hold:
\begin{itemize}
	\item[(i)]For every $\eta>0$ and rational $t\geq 0$, there exists a compact set $\Lambda_{\eta, t} \subset E$ such that
	$$\liminf_{n\rightarrow \infty}\mathbb{P}\{X_n(t) \in \Lambda_{\eta,t}\} \geq 1-\eta.$$
	\item[(ii)] For every $\eta>0$ and $T>0$, there exits $\delta>0$ such that
	$$\limsup_{n\rightarrow\infty} \mathbb{P}\{ w'(X_n, \delta, T) \geq \eta \} \leq \eta.$$
\end{itemize}
where we have used the \textbf{modulus of continuity} $w'$ defined as follows: for $x\in D([0,\infty)\times E)$, $\delta>0$, and $T>0$:
$$w'(x,\delta, T) = \inf_{\{t_i\}} \max_i \sup_{s,t \in [t_{i-1}, t_i)} r(x(s), x(t)),$$
where $\{t_i\}$ ranges over all partitions of the form $0=t_0<t_1<\hdots< t_{n-1}<T\leq t_n$ with $\min_{1\leq i\leq n}(t_i-t_{i-1})>\delta$ and $n\geq 1$
\end{theorem}

\begin{theorem}[Continuity criteria for the limit in Skorokhod spaces (\cite{ethier2009markov}, Theorem 10.2, Chapter 3)]
\label{th:continuity_criteria_limit_Skorokhod_space}
Let $(E,r)$ be a metric space.
Let $X_n$, $n=1,2,\hdots,$ and $X$ be processes with sample paths in $D([0,\infty);E)$ and suppose that $X_n$ converges in distribution to $X$. Then $X$ is a.s. continuous if and only if $J(X_n)$ converges to zero in distribution, where
$$J(x) = \int^\infty_0 e^{-u} [ J(x,u) \wedge 1] \, du$$
for 
$$J(x,u) = \sup_{0\leq t \leq u} r(x(t), x(t-)).$$
\end{theorem}

\bibliographystyle{plain}

\end{document}